\documentclass[11pt,oneside]{article}

\usepackage[utf8]{inputenc} 
\usepackage[T1]{fontenc}    
\usepackage{amsmath}
\usepackage{amsfonts}
\usepackage{amsthm}
\usepackage{amssymb}
\usepackage{color}
\usepackage{hyperref}       
\usepackage{url}            
\usepackage{graphicx}
\usepackage[round]{natbib}
\usepackage{doi}
\usepackage[a4paper]{geometry}

\usepackage{mathabx}

\theoremstyle{definition}
\newtheorem{definition}{Definition}[section]
\newtheorem{example}[definition]{Example}
\theoremstyle{plain}

\newtheorem{proposition}[definition]{Proposition}

\newtheorem{lemma}[definition]{Lemma}

\theoremstyle{remark}
\newtheorem{remark}[definition]{Remark}

\newcommand{\R}{\mathbb{R}}

\newcommand{\N}{\mathbb{N}}

\newcommand{\EE}{\mathbb{E}}

\newcommand{\cov}{{\rm cov}}
\def\lnorm{\left\ldbrack}
\def\rnorm{\right\rdbrack}

\def\A{{\cal A}}
\def\B{{\cal B}}

\def\X{{\cal X}}

\def\rank{\text{rank}\,}

\def\RR{\R} 
\def\NN{\N} 
\def\Var{\text{Var}}
\def\Cov{\text{Cov}}

\def\PP{{\mathbb P}}

\def\basis{{\cal U}}

\def\1{{\bf 1}}

\long\def\notiz#1{{\bf #1}}
\def\notiz#1{}

\def \sign{{\rm sign\,}}

\def\GF{{\color{red}O}}
\def\GX{{\color{blue}U}}
\def\text#1{\mbox{\rm #1}}
\def\plusG{{\color{red} \dotplus}}
\def\dotsumG{\color{red}\sum\limits^{.}}
\def\plus{{\color{blue}\ddot+}}

\def\dettrafo{\GF_t}
\def\detsys{\GF_s}

\def\aopt{a_{\rm opt}}
\def\aopt{a_{\signXi}}
\def\rhoCanon{\rho_{\rm ca}}

\long\def\note#1{{\color{red}#1}}
\long\def\note#1{ }

\def\GF{G}
\def\GFsup{\GF^*}

\def\GX{H}
\def\plusG{\ddot+}
\def\dotsumG{\sum\limits^{.}}
\def\plus{\dotplus}

\def\Pois{{\rm Poiss}}
\def\signXi{\sign {\lnorm \cdot \rnorm}}
\def\assoc{{a}}

\allowdisplaybreaks
\begin{document}
\title{
 An algebraic generalization of the entropy
  \\and its application to  statistics}
\author{Martin   Schlather\footnote{
    University of Mannheim, Institute of Mathematics, 68159 Mannheim, Germany. Email:
    martin.schlather@uni-mannheim.de}
} 
\date{\today} 
\maketitle

\begin{abstract}
  We define a general notion of entropy in elementary,
  algebraic terms.
  Based on that, weak forms of a scalar product and a distance
  measure are derived.
  We give basic properties of these quantities,
  generalize the Cauchy-Schwarz inequality,
  and relate our approach to the theory of scoring rules.
  Many supporting examples illustrate our approach and give new
  perspectives on established notions, such as the
  likelihood, the Kullback-Leibler divergence, the uncorrelatedness of
  random variables, the scalar product itself, the Tichonov regularization
  and the mutual information.
 \end{abstract}


\smallskip

Classification:
20N99, 
16Y99, 
94A17, 
46C99, 
62A01,  
62R01 

\section{Introduction}

The original motivation of this research work was
the transfer of the principal component analysis (PCA)
to extremes. An obvious obstacle is that PCA requires
the existence of second moments, whereas, in extreme value theory, not
even the expectation is supposed to exist.
Only a few papers deal with this problem,
directly \citep{cooleythibaud19,dreessabourin21} or indirectly \citep{fissler2023generalised}, 
but none of them aimed for a unified approach to PCA.
The search for this unification revealed a rather general underlying algebraic structure to
statistical modelling, which we present here.
In order to understand the nature of this structure we tried to
reduce the assumptions to the essential.
Our paper can be considered as a complement to
\cite{mccullagh2002statistical}, as it lays the foundations for 
concrete constructions for model adaptations
 through minimization problems of certain risk functions.

The introduced concept is rather general, hence very flexible
concerning its application and interpretation.
Our explicit examples stem from very different areas, so that we will
use the physical notion of a system to describe 
a single random variable,
an information channel or a point in a vector
space. We have also applications to stochastic processes and molecular systems
in mind, but treat here  only very simple
examples.
Besides its generality, the gain of this article
 is a novel understanding
of established concepts, such as  the uncorrelatedness of two random
variables, an extension of such concepts without the necessity of a
second moment,
and the presentation of starting points for solutions to current problems, such as the
PCA for extreme values.
As the paper combines terms and concepts from various fields, including
algebra, information theory, physics, analysis and statistics, we henceforth
introduce into all relevant aspects.

Whereas the Gaussian world deals with sums of random variables,
extreme value theory deals with maxima. Since the binary operator ``$\max$''
does not have an inverse element \citep{maclagan2021introduction},
we will rely on monoids, i.e.,
semigroups with a neutral element,
as basic structures.
We will go further and weaken the algebraic laws
in Section \ref{sec:single}.
As neutral elements are important to our approach,
neutral elements have to be included also on the statistical side,
e.g., the Dirac measure $\delta_0$
will be part of
the family of centered Gaussian distributions.
Connections between statistics and algebra are well-established.
Besides the definition of a statistical model
\citep{mccullagh2002statistical},
algebraic approaches to statistical problems
have turned out to be useful
in experimental design, 
non-parametric confidence areas \citep{francis2017building},
investigation of the likelihood \citep{manivel2023complete},
graphs \citep{dehmer2011history},
and quantum probability \citep{wigner31}. 

The physical entropy is, according to Wikipedia, ``a scientific
concept that is most commonly associated with 
a state of disorder, randomness, or uncertainty'',
cf.~\cite{tsallis2022entropy} and \cite{ilic2021overview}.
We abstract from this definition and consider an entropy
as a mapping that is additive for independent
systems and non-negative.
From this point of view, the squared norm of a pre-Hilbert space
will be considered as an entropy. 
Also the  variance  is considered as a measure of entropy
for random variables with existing second moments, cf.\
\cite{nguyen2019varphi}.
A further example is the
probability measure itself, if
disjointness of two events is understood as
independence. 
The Shannon entropy and more general the R\'enyi entropy,
both information theoretical
measures of uncertainty, are also included in the generalized definition.
In contrast to usual approaches in statistics to the notion of entropy, we do
not include any kind of randomness in its definition;
the proposed connection between our definition of the entropy and the
statistical 
model is postponed to the last
section.
Note that additivity and positivity
of the entropy is not always assumed in literature,
cf.\ the superadditivity in \cite{kaniadakis2017composition} or  the differential entropy.
Also the Tsallis entropy \citep{havrda1967quantification, tsallis88} is
considered as
non-additive in literature, in contrast to our point of view in Example \ref{ex:tsallis}.
The idea of an algebraic approach to the entropy 
is not new.
For instance,
\cite{weiss74} 
considers the entropy
of an endomorphism
over an Abelian group,
see also  \cite{dikranjan2009algebraic}.
Proposition 1.4 there is similar to our 
Definition \ref{def:entropy.hemi-group} and the construction
in our Example  \ref{ex:productspace}.
\cite{bellon1999algebraic} define an entropy for rational maps.
\cite{hou2016entropy} consider semigroups of continuous
endomorphisms on a compact metric space.
The ideas behind the approach of \cite{tempesta16},
see also \cite{curado2016new}, are rather close to our Definition
\ref{def:entropy}, although the details and the directions differ.
Similar to the point of view in \cite{tsallis2016approach}, we
consider the entropy as a non-unique quantity, which may become unique
under certain additional conditions.

The notion of the scalar product is closely connected to the pre-Hilbert space,
although generalizations exist, in particular to the semi-inner-products
for a Banach space. In
most approaches, the basic property of linearity is at least partially
kept, see 
\cite{dragomir2004semi}, \cite{ehm2004kleene},
\cite{istratescu2012inner}, \cite{litvinov2000linear},
\cite{lumer1961dissipative}, \cite{nasehpour2018valuation},
\cite{tan14inner}, to cite
of few.
Our  Definition \ref{def:scalarproduct} of a generalized scalar
product 
is
geometrically motivated and measures the
deviation from the Pythagorean theorem.
A Pythagorean relation is considered important in statistics
\citep{efron1978geometry,stehlik2012decompositions}.
In our approach, the generalized
scalar product possesses only a rather weak property of linearity,
cf.\ Eq.\ \eqref{eq:sp.linear}.
Special cases of the generalized scalar product
include the scalar product
of the pre-Hilbert space and the probability of the intersection of two sets.
The generalized scalar product appears  in a
 generalized metric (Section
\ref{sec:comparing}),
which includes the geostatistical variogram \citep{chilesdelfiner}
and the variation of
information.
The Kullback-Leibler divergence 
can be interpreted as both a scalar
product and a generalized metric, cf.\ Example \ref{ex:KL}.

Being the basis of any statistical inference,
the definition of a statistical model $\cal P$
as a family of distributions 
is taught in every basic course of mathematical statistics.
At an advanced level, the additional properties, which
a statistical model should
have, are not obvious.
\cite{mccullagh2002statistical} requires the consistency of parameter
estimation for nested models.
\cite{drton2007algebraic} require an additional algebraic property,
namely that a subset of the parameter
space is a semi-algebraic set.
\cite{helland2004statistical,helland2018statistical} considers invariant
transformation groups within the parameter space. 
Helland's approach comes rather close to ours, since
 we abstract the set of probability distributions
$\cal P$ to an algebraic system of transformations that can be applied
to a given distribution (Section \ref{sec:stat.model}).
Similar to Helland's approach, our
examples given below indicate that
certain subsets of transformations can have much stronger properties.
Solely 
the properties of the transformations
are relevant for the statistical inference, so that
different families of
distributions
might be treated by the same approach.

We postulate here, that
a measure for the entropy shall fulfill three different tasks, which we consider
in the following sections separately for clarity:
\begin{enumerate}
\item When two independent systems are joined without interaction,
  their entropies shall add \citep{penrose2005foundations}, cf.~Definition \ref{def:entropy.hemi-group}.
\item
  With regard to optimization problems including model adaptations, 
 we shall be able to compare two  systems by a
 a risk function or generalized metric. In Section~\ref{sec:comparing},
 we suggest
 a method 
 how to determine  the amount of entropy attributed to the discrepancy
 of two systems.
\item In regard of physical and statistical applications, 
  we shall be able to transform systems in a deterministic way and
  to determine the entropy of the transformed system (Definition
  \ref{def:entropy.hemi-ring}).
\end{enumerate}
Until Section \ref{sec:stat.model}, an underlying stochastic
framework is not necessary.
Only Definition \ref{def:stat.model}  bridges then to stochastic
systems on a theoretical level.
Nonetheless, many illustrating examples, which will
be presented in separate sections,
will deal with explicit
stochastic models.
With regard to the length of the paper 
and the self-containedness of the 
results in Section \ref{sec:comparing},
Sections \ref{sec:trafo} and
\ref{sec:stat.model} are rather short and
lay only the algebraic fundamentals for a broad range
of statistical applications. In particular, Section
\ref{sec:stat.model} is considered as an outlook,
which flashes on the direction only.
\cite{schlatherreinbott21}  give examples of statistical applications.


\section{Entropy-driven algebraic approach}
As an entropy specifies the amount of randomness or noise in a
system, it shall be
non-negative.
Furthermore, small changes in randomness shall lead to a small change
in the entropy measure, in general. So, we need an underlying topology.

\subsection{Non-interacting  entropy-driven systems} \label{sec:single}
\label{sec:valuation}

\begin{definition}\label{def:entropy}
  Let $\GF$ be a topological space and
  $\lnorm . \rnorm : \GF \rightarrow
  [0,\infty]$ be a  map.
  Assume that the set of
  zero elements $\GF_0 := \{\xi \in \GF :  \lnorm \xi
  \rnorm =0\}$ is measurable, neither empty nor the whole space $\GF$.
  If  the map $\lnorm . \rnorm$ is continuous on $\GF \setminus \GF_0$,
  then
    $\lnorm . \rnorm $ is
    called  an \emph{entropy measure} for $\GF$ (or an \emph{entropy} on $\GF$).
  \end{definition}

\begin{definition}\label{def:entropy.hemi-group}
  Let $\GFsup$ be a topological space and
  $\lnorm . \rnorm : \GFsup \rightarrow
  [0,\infty]$ a  map.
  Let
  $\GF \subset \GFsup$ be a
  measurable subset and
  $\circ:\GF \times \GFsup \rightarrow  \GFsup$
  a measurable
  operator.
  If   $\lnorm . \rnorm$  restricted to  $\GF $ is an entropy and
  $$
  \lnorm \xi \circ \nu \rnorm = 
  \lnorm \xi  \rnorm +
  \lnorm \nu \rnorm 
\quad
\forall \xi\in \GF,\nu \in\GFsup
  ,$$ 
then we call
  $(\GF,\circ, \lnorm.\rnorm)$ an \emph{entropy-driven hemi-group}.
\end{definition}

\begin{remark}\label{rem:hemi}
  \begin{enumerate}
  \item
    We use the term hemi-group to
    underline that the imposed structure is less than that of a
    semigroup. 
    If $(\GF,\circ)$ is a semigroup and the entropy is continuous,
    then the entropy is typically
    a homomorphism into $([0,\infty], +)$.    
  \item
    We will understand $ \lnorm \xi \circ \nu \circ \eta \rnorm$
    as  $\lnorm \xi \circ (\nu \circ \eta) \rnorm$ for
    $\xi,\nu\in\GF$, $\eta\in\GFsup$.
    If $\GFsup \not= \GF $, then $\GFsup \setminus \GF $ plays the
    role of a bin. We are interested only in $\GF $, but
    will give some explicit examples of $\GFsup$ in
    Subsection~\ref{sec:ex:entropy}.
 \item \label{rem:hemi.hemi}
    The operator $\circ$  might be called left (right)
    hemi-commutative 
    in the sense that for all $\xi,\nu\in\GF$ we have
    \begin{eqnarray*}
      \lnorm \xi \circ \nu \circ \eta \rnorm
      = 
      \lnorm  \nu \circ \xi \circ \eta \rnorm\; \forall \eta\in\GFsup
       &\mbox{and}&
    \lnorm \xi \circ \nu \circ \eta \rnorm
      = 
                    \lnorm  \xi \circ \eta  \circ \nu\rnorm
                    \; \forall \eta\in\GF
                    ,
    \end{eqnarray*}
    respectively.
    Furthermore, an entropy-driven hemi-group $(\GF ,\circ)$
    has at least one left (and right) hemi-neutral element.
  \item
    Elements in $\GF_0$ can have a behaviour that is rather different
    from that of all other elements. In a pre-Hilbert space, the vector $0$
    cannot be renormalized to a vector of length $1$.
    In quantum computing this is important as the impossibility of
    renormalization  signifies the end
    of the program. Furthermore,
    Proposition \ref{lemma:reverse.scalar} shows that $\GF_0$ plays
    a special role, when reconstructing an entropy measure from a
    generalized scalar product. A last example is the geostatistical
    variogram, where important specifications have a jump in
    the origin and are continuous elsewhere \citep{gneitingsasvari99}.
    
  \end{enumerate}
\end{remark}

\subsection{Examples of entropy-driven systems}
\label{sec:ex:entropy}

Let $\NN_0 = \NN \cup \{0\}$. We denote a genuine subset by $\subsetneq$.

\begin{example}\label{ex:shannon}
  The Shannon entropy has been developed
  for characterizing and optimizing the transmission of
  information.
  Let $\alpha$ be some finite alphabet.
  It is usually assumed that the sequentially transmitted letters are
  independent and identically distributed according to a
  discrete distributon $p$ given by $p_x$, $x\in\alpha$.
  The Shannon entropy $\lnorm p \rnorm$ equals
  $$
  \lnorm p \rnorm = - \sum_x p_x \log p_x
  .$$
  Since the transferred letters are naturally modelled by a concatenation
  semigroup, the Shannon approach transforms 
  into a trivial concatenation
  semigroup of probabilities
  $(\GF , \circ)$ with
  $\GF  = \{ p^n : n\in\NN_{0}\}$  and
  \begin{eqnarray}\label{eq:shannon.p}
     \lnorm p^n \rnorm = - n\sum_x p_x \log p_x,\qquad n\in\NN_{0}
  .
  \end{eqnarray}
 To see clearer the entropy approach in Equation
 \eqref{eq:shannon.p}, we assume for simplicity that $n=2$ and
 that the second letter $y$ follows a different distribution, $q$, say.
 We still assume that the two letters are independent.
 As we only assume that the alphabet is finite, we may now extend the
 alphabet and include all
 pairs of letters  so that we get for the joint
 distribution~$r$  of
 the pairs
 $$
 \lnorm r \rnorm
 = -\sum_{xy} r_{xy}  \log r_{xy}
 = - \sum_x \sum_y (p_x q_y)
\log(p_x q_y)
= \lnorm p  \rnorm  + \lnorm q \rnorm 
.$$
Similar considerations apply to the more general R\'enyi entropy.
Both, the Shannon entropy and the R\'enyi entropy assume
that information is transmitted unbrokenly.
Example \ref{ex:shannon.dependable} suggests a model for transmission errors
in the case of the Shannon entropy.
  
 \end{example}

\begin{example}\label{ex:basic}
  In  stochastics, the binary operator $\circ$  models the
  ``merge'' of two independent random
  variables.
  \begin{enumerate}
  \item In the context of the central limit theorem, ``merging''
    signifies the addition of random variables.
    In particular, consider
    the set of centered Gaussian distributions
    including the Dirac measure $\delta_0$.
    Here, the variance is a (or better the) suitable entropy measure
    $\lnorm \cdot \rnorm$, see
    Example \ref{ex:stable.stat} for a rational.
    Instead of
    considering the family of distributions or the family of random
    variables, we may alternatively consider its
    parameter space $\GF $  of standard deviations, i.e.,
    $\GFsup = \GF  = [0,\infty)$.  Hence,
    $\lnorm s \rnorm = s^2$ for $s\in \GF $,
     $\GF_0 = \{ 0 \}$ and 
    $s \circ t = \sqrt{s^2 + t^2}$.
    Note that, if we were not interested in adding random variables,
    but in the amount of information a single random variable contains, we
    would choose the differential entropy as entropy measure
    together with the underlying concatentation semigroup.
    The differential entropy is the limit of differences between two Shannon
    entropies, the one which is of interest and a suitably choosen
    uniform distribution. Hence, the differential entropy can be
    negative and is therefore beyond the scope of this paper.
    
  \item
    The Fr\'echet distribution  $\Phi_{\alpha,\lambda}$ is given by
    $\Phi_{\alpha,\lambda}(x) = \exp( - \lambda^\alpha / x^\alpha)
    \1_{x > 0}$, $\alpha,\lambda>0$. 
    It appears as a limit distribution of normalized maxima of
    independent, identically distributed random variables.
    For a fixed $\alpha > 0$ consider the set of all $\Phi_{\alpha, \lambda}$,
    $\lambda \ge 0$. Then, the parameter space is again $\GFsup=\GF =[0,\infty)$,
    and $\lnorm \lambda \rnorm = \lambda^\alpha$
    turns out to be the right choice, hence
    $\lambda \circ \mu = (\lambda^\alpha +
    \mu^\alpha)^{1/\alpha}$.
    See Examples \ref{ex:realaxis}.\ref{ex:frechet.1}
    and  \ref{ex:maxstable.stat} for further information.
    
  \end{enumerate}
\end{example}

We finalize this section with a further example from stochastics.

\begin{example}\label{ex:productspace}
  Let $\GFsup = \left(\times_{i=1}^\infty \Omega, \otimes_{i=1}^\infty \A,
   \times_{i=1}^\infty \PP\right)$ be a product 
  probability space. Let $\Omega^\infty = \times_{i=1}^\infty \Omega$,
  $\GF  = \{ A \times \Omega^\infty : A \in \A\}$ and
  $$
  \lnorm A \times  \Omega^\infty\rnorm := -\log
  \PP(A)
  $$
  for all $A \in \A$.
  This choice is called information content of $A$
  \citep{mcmahon2007quantum}.
   Let $\circ$ be defined by
  $$ \left(\bigtimes_{i=1}^m A_i \times \Omega^\infty\right) \circ
  \left(\bigtimes_{j=1}^n B_j \times \Omega^\infty\right)
  \rightarrow  \bigtimes_{i=1}^m A_i\times
  \bigtimes_{j=1}^n B_j \times \Omega^\infty
  \qquad \forall m,n\in\NN, A_i,B_j\in \A
  $$
  with $A_m \not= \Omega$, $B_n\not=\Omega$.
  Then, $(\GF , \circ, \lnorm\cdot  \rnorm)$ is an entropy-driven
  hemi-group.
\end{example}

\section{Entropy-driven systems with comparable elements}
\label{sec:comparing}

The approach in the previous
Section \ref{sec:valuation} is sufficient to model  simple situations,
which involve independent systems.
To compare two systems,
for instance in order to find the best approximation of a system by an element
of a given class of systems,
we need
a kernel $\rho$,
which defines the distance between
two systems.
A kernel is a real-valued function that depends on two variables of the same set of systems.

We call $\rho$ a hemi-metric (or risk or hemi-divergence as it
generalizes the statistical divergence).
In a statistical context,
the hemi-metric will measure the entropy
of the residuals
between the original data $X$, say, and its approximation,
$Y$.
The definition of the hemi-metric $\rho$ will be based on a
second, binary operation denoted by $\plus$, which 
models the interacting merge of two 
systems and which is
the standard
addition in many practical applications.
The motivating idea is that the entropy of a residual
equals the entropy of the part of $X \plus Y$
that cannot be explained  by $X \plus Y^*$
where $Y^*$ is a copy of $Y$ that is independent of $X$. 
To be more specific, let $X$ and $Y$ be real-valued random variables with
finite second moments. Then,
\begin{eqnarray}
  \label{eq:var}
  \Var(X - Y) = -\Var(X + Y) + 2\Var (X + Y^*)
  .
\end{eqnarray}
Obviously, the left handside of \eqref{eq:var} needs an underlying
group structure on rather complicated objects $X$ and $Y$, while the
right handside needs at most a semigroup structure for $X$ and $Y$ and
exploits the
group structure of $\RR$ instead.
Hence, the right handside can be generalized to
binary operators $\plus$ that do not allow for inverse elements.
As we have an entropy based point of view,
we shall require that
\begin{eqnarray}
  \label{eq:rho.entropy}
  \rho(\xi, \eta) = \lnorm \xi \rnorm
\quad\
\forall \xi\in \GF, \eta\in \detsys
,
\end{eqnarray}
for some non-empty subset $\detsys\subset  \GF_0$
of the zero elements. 
We consider the set $\detsys$ as the set of deterministic systems.
Combining \eqref{eq:var} and \eqref{eq:rho.entropy},  our proposed ansatz is 
\begin{eqnarray}\label{eq:GFGS}
  \lnorm \xi \plus \eta \rnorm &= & \lnorm\eta \plus \xi  \rnorm=\lnorm \xi
  \rnorm,\qquad \forall  \xi \in \GF, \eta \in \detsys,
  \\
  \rho_a(\xi, \nu)
  & =&\label{eq:rhoc}
       a \lnorm \xi \plus \nu\rnorm + (1-a) \lnorm \xi \circ \nu \rnorm
,\quad \xi,\nu\in \GF,
\end{eqnarray}
that is, the kernel $\rho_a$ consists of  an affine combination of
entropies with $a\in\RR$.
Since $\rho_a$ shall be interpreted as a measure of
entropy, we define
$$
\Xi = \{ a \in \RR  : \rho_a(\xi,\nu)  \ge  0\quad \forall \xi,\nu \in
\GF
\}
$$
and require that $a\in\Xi$. Obviously, $\Xi$ is an interval which
includes $[0,1]$.
Let
\begin{eqnarray*}
  M_\GF &=& \sup_{\xi,\nu \in \GF, \lnorm \xi \circ \nu
          \rnorm > 0} \frac{ \lnorm \xi  \plus \nu\rnorm}{\lnorm \xi \circ \nu 
          \rnorm}
          ,
  \\
  m_\GF &=& \inf_{\xi,\nu \in \GF, \lnorm \xi \circ \nu
          \rnorm > 0} \frac{ \lnorm \xi \plus \nu\rnorm}{\lnorm \xi \circ
          \nu
          \rnorm}
          .
\end{eqnarray*}
so that $m_\GF \in[0,1]$ and $M_\GF \in[1, \infty]$, since
$\lnorm \xi \plus  \nu\rnorm = \lnorm \xi  \rnorm=
\lnorm \xi  \circ \nu \rnorm$ for all $\xi\in\GF$ and
$\nu\in\detsys$. 
Hence,
$$
\Xi = [(1-M_\GF )^{-1},(1-m_\GF )^{-1}] \setminus\{-\infty, \infty\}
.
$$
In particular,
$$
\Xi = [(1-M_\GF )^{-1}, 1] \setminus \{ -\infty\},
$$
if $(\GF ,\plus)$ is a group.
From a statistical point of view, the quotient  
$\lnorm \xi  \plus \nu\rnorm / \lnorm \xi \circ \nu  \rnorm$ can be
interpreted as a relative noise content.

\begin{definition}\label{def:rho}\label{def:plus} \label{def:scalarproduct}
  Let $\GF$ be an entropy-driven hemi-group with $\lnorm \xi \rnorm <
  \infty$ for all $\xi\in\GF$.
  Assume, a binary operation $\plus$ on $\GF$ exists with 
  property \eqref{eq:GFGS} and
  the function
    \begin{eqnarray}
      \label{eq:cond2.comparable}
    \xi \mapsto \lnorm \xi \plus \xi \rnorm - 2\lnorm \xi \rnorm     
    \end{eqnarray}
    shall not change sign.
    The sign is denoted by $\signXi$ with $\signXi\in\{-1,+1\}$.
  Then, the tuple  $(\GF,\circ, \plus, \lnorm .\rnorm)$ is said to have
  \emph{comparable elements}.
  A kernel $\rho_a$, given by \eqref{eq:rhoc} with $a\in\Xi$, is
  called \emph{associated hemi-metric}.
  The kernel
  \begin{eqnarray}\label{eq:sp.assoc}
 \langle ., .\rangle_{\assoc} : \GF \times \GF \rightarrow \RR,\quad
  (\xi,\eta) \mapsto \signXi (
   \lnorm  \xi  \plus \eta \rnorm
   -
\lnorm \xi \circ\eta \rnorm
  )    
  \end{eqnarray}
  is called the \emph{associated hemi-scalar product}.
  If $ \langle \xi , \eta\rangle _\assoc= 0$ ($>0$, $<0$), we call
  $\xi$ and $\eta$ \emph{orthogonal}/\emph{uncorrelated} (\emph{positively correlated},
  \emph{negatively correlated}).
 \end{definition}

   In quantum stochastic calculus, our  
   redefinition of the scalar product is
   appreciated in the special case of the covariance $\cov(X,Y)$,
    as it is based on the observables $X$, $Y$, $X+Y$ only,
    hence accessible to measurements    
    \citep[p.\ 15]{parthasarathy}.
The  idea of the hemi-scalar product also appears implicitly
 in  Proposition 3.2 in Chapter 3 of \cite{BCR}, 
which    deals with the  generalization of a
squared difference of real values towards complex values, in the framework
of Hilbert spaces.

The choice of an affine combination in \eqref{eq:rhoc} is justified by
the fact that $\rho_a$ appears frequently in practical
applications, see Subsection \ref{sec:ex:comparable}, in particular
Example \ref{ex:tichonov}.
Since, in many practical
examples, the entropy measure is a convex or concave function
and since all convex combinations are included in $\Xi$, the
hemi-measure $\rho_a$ has a chance to inherit this property at least
for some $a\in\Xi$.
From an aesthetic point of view, the affine functions are simple;
they 
are the only functions that are both
concave and convex, and their curvature is minimal.

    \note{
    \cite[p.\ 15]{parthasarathy}:
``It is noted that for observables $X$, $Y$, whose coariance in a
    state $\rho$ is real one has
    $$
    \Cov_\rho(X, Y) = \frac12\left(\Var_\rho(X+Y)-\Var_\rho X -
      \Var_\rho(Y)\right) 
    .$$
    One may use the right hand side of this equation as an alternative
    definition of covariance. This has the advantage that it can be
    estimated by repeated individual measurements on the observables
    $X$, $Y$ and $X+Y$.''
}

  Note that the above definition reverses standard definitions and
  corollaries. For instance, in the standard approach of
  stochastics, uncorrelatedness is
  defined as the covariance being $0$, hence
  the variance of the sum of uncorrelated second-order
  random variables equals the sum of the variances.
  Here, Definition \ref{def:scalarproduct} 
  suggests the interpretation that
  two random quantities are called uncorrelated if they behave as if they
  were independent (with respect to their entropy).
  Or, in the wording of a pre-Hilbert space,
  the equality $ \lnorm  \xi  \plus \eta \rnorm=\lnorm \xi \circ\eta \rnorm
 $ is the statement of the Pythagorean theorem, i.e.,
  we call two vectors orthogonal if
  the Pythagorean theorem holds (and not reversely, if two vectors are
  orthogonal, then the Pythagorean theorem holds).
  Apart from the interpretation of the hemi-scalar product as the deviation
  from the Pythagorean theorem, some other issues shall be mentioned.
  First, the elementary requirement of a standard scalar product, that
  it is bilinear, has been dropped in the definition of the
  hemi-scalar product. Bilinearity is now a pecularity of a pre-Hilbert
  space.
  Second, motivated by statistical applications,
  the definition is chosen such that the correlation with
  itself is never negative; this excludes an immediate application
  of our approach to Krein spaces, for instance.
  Third, in contrast to common approaches in statistics, the residual
  itself remains undefined; we only know the entropy of this fictitious
  residual.
  Fourth, our approach does not guarantee that $\rho_a(\xi, \xi) = 0$
  for all $\xi \in \GF$.
  Fifth, our approach cannot exclude that
  $ \lnorm \xi \plus \xi \rnorm - 2\lnorm \xi \rnorm \equiv 0$, i.e., in
  some cases,  $\signXi$ is undetermined and can be chosen freely, cf.~Example
  \ref{ex:realaxis}.\ref{ex:cauchy}. 

\begin{remark}
  Let $a\in\Xi \setminus\{0,1\}$ and $\xi,\eta \in \GF$. Then, the following assertions are equivalent:
  \begin{enumerate}
  \item  $\xi$ and $\eta$ are uncorrelated;
  \item  $ \lnorm \xi \circ\eta \rnorm = \lnorm
    \xi  \plus \eta \rnorm$;
  \item $\rho_a(\xi, \eta) = \lnorm \xi \circ\eta \rnorm$;
  \item $\rho_a(\xi, \eta) = \lnorm \xi \plus\eta \rnorm$.
  \end{enumerate}
 \end{remark}

\begin{remark}
  The definition of the scalar product  suggests a multivariate extension,
  $\langle \xi_1,\ldots, \xi_n\rangle_\assoc= \signXi  \left(\lnorm \dotsumG{}_{i=1}^n \xi_i \rnorm -
  \lnorm  \bigcirc_{i=1}^n \xi_i \rnorm 
\right)$, so that an ordered set $\{\xi_1,\ldots, \xi_n\}$ might be called
uncorrelated
if   $\lnorm \dotsumG{}_{i=1}^n \xi_i \rnorm =
  \lnorm  \bigcirc_{i=1}^n \xi_i \rnorm $.
\end{remark}

\begin{definition}\label{def:comparable} 
  Assume,  that  the conditions in Defintion \ref{def:rho} are satisfied.
  Let $a_{-1}= 1 / (1- m_\GF) $ and $a_{+1}=1 / (1- M_\GF)$, and
  assume that $ \aopt$ is finite. 
  Then, the hemi-metric $\rhoCanon:=\rho_{\aopt}$  is called \emph{canonical}.
  The two
  kernels
  \begin{eqnarray*}
    \langle .,
    .\rangle &:=  &|\aopt|\, \langle ., .\rangle_{\assoc}
                    = -\aopt( \lnorm . \plus. \rnorm
                     -  \lnorm . \circ .\rnorm )
    \\
 \langle ., .\rangle_{2} & :=&  \langle .,
  .\rangle / 2
  \end{eqnarray*}
   are called \emph{canonical hemi-scalar products}.
  
 \end{definition}

While $ \langle ., .\rangle_2$ is often used in practice, 
$ \langle ., .\rangle$ suits better 
 this set-up. In particular, we have
\begin{eqnarray}\label{eq:rhoopt}
  \rhoCanon(\xi, \eta)
  & =& \lnorm \xi \circ \eta \rnorm - \langle \xi, \eta \rangle
\end{eqnarray}
and
\begin{eqnarray}
  \rhoCanon(\xi, \eta)    &= &\label{eq:rhoopt2}
          \lnorm \xi \plus \eta \rnorm - \frac{\aopt-1}{\aopt }
         \langle \xi, \eta \rangle
         \qquad
         \mbox{ if }
         \aopt \not=0
 .  
\end{eqnarray}

  A superficial reasoning for the choice of $\aopt$ is that the choice $a=0$
  is not sound in most practical applications;
  choosing a value far away from $0$ looks hence promising.
  A deaper reasoning is the following. In practical applications, an
  observation $\xi$, say, consists of a structural part and a noisy
  part. A model $\eta$ is searched such that the amount of noise
  $\rho(\xi,
  \eta)$ between the data and the model is minimized. Among the
  ensemble of risk functions, the one which gives the smallest values
  to the noisy part seems to be outstanding.
  For any $a\in\Xi$ the fraction of the
  noisy part, i.e., the part of $\xi \plus\nu$
  that cannot be explained 
  by $\xi \circ \nu $,
  can be defined as
 $
 \rho_a(\xi ,\nu)/({a \lnorm \xi  \plus \nu\rnorm})
 $
 for $a \lnorm \xi  \plus \nu\rnorm\not = 0$.
 Then, its  modulus
 $f_{\xi,\nu}(a)  := \rho_a(\xi ,\nu)/({|a| \lnorm \xi  \plus \nu\rnorm})$
 should be minimized.
 For any $\xi,\nu\in\RR$ with 
 ${\lnorm \xi \circ \nu 
   \rnorm}  / { \lnorm \xi  \plus \nu\rnorm} \in (0,\infty)$, 
the function  $f_{\xi,\nu}$
 is strictly increasing for negative arguments and decreasing for positive
 arguments. Hence, 
 Definition
 \ref{def:comparable} is a consequence.
\note{
 Furthermore,
 $f_{\xi,\nu}(\pm \infty) =\pm( 1 - \lnorm \xi\circ \nu \rnorm / \lnorm
 \xi\plus \nu \rnorm)$,
 and $f_{\xi,\nu}(a) \ge 0$ for all $\xi,\nu\in\GF$ with $\lnorm
 \xi \plus \nu\rnorm >0$ and $a\in\Xi$.

Furthermore, $f_{\xi,\nu} (a) \ge 1 + (a^{-1} -1) / m_\GF $ for $a\ge 1$
so that  $\sup \Xi =   1 / (m_\GF  -1)$, and  
$f_{\xi,\nu} (a) \ge -1 + (1 -a^{-1}) / M_\GF $ for $a\ge 1$ for $a<0$
so that $\inf \Xi = 1  / (M_\GF  -1)$.
}


\begin{remark}\label{bem:infty}
  Definition
  \ref{def:comparable} prefers $\rho_{\aopt}$ over
  $\rho_{a_{-\signXi}}$ as many examples follow this scheme.
  The choice of $\rho_a$
  is however not that clear cut. 
  Particularly,
  if either $\inf \Xi = -\infty $ or $\sup\Xi = \infty$,
  i.e., $M_\GF=1$ of $m_\GF=1$,
  an alternative
   to $\rhoCanon$ is
   $$
   \rho_{\infty} := \lim_{a\rightarrow \pm \infty} \frac{\rho_a}{|a|}
   = \langle \cdot, \cdot \rangle_\assoc
   .
   $$
   Note, that  $\rho_\infty$ is necessarily identically $0$,
   if both $\inf \Xi = -\infty $ and $\sup\Xi = \infty$.
 \end{remark}

 \subsection{Examples}
 \label{sec:ex:comparable}

From an algebraic point of view, it is desirable to put much more
structure onto the operator $\plus$ than required in its definition.
This will be done  in the next subsections.
The examples given
here are not targeted to this point.

\begin{example}[Pre-Hilbert space]\label{bsp:hilbert}
  Let $H$ be a pre-Hilbert space with scalar product $\langle \cdot ,
  \cdot \rangle_H$.
  Let $\lnorm x \rnorm := \langle x, x\rangle_H$ and  $\plus$
  be  the standard addition in $H$.
  We define $\circ : H \times H \rightarrow H^2$
  as the map $(x; y)\mapsto (x, 0) + (0,y) = (x,y)$.
  To be consistent with Example \ref{ex:productspace}, we might
  alternatively let $x \circ y = (y, 0)$, if $x=0$.
  Now, $\signXi =+1$
  and we have for $x,y\in
  H$ with $\lnorm x \circ y \rnorm \not=0$, that 
  $$
  \frac {\lnorm x + y \rnorm}{\lnorm x \circ y \rnorm}
  \le 1 + 2 \sup_{z\in H} \frac{ \langle e, z\rangle_2}{1 +
    \|z\|^2} =
  1  + 2 \sup_{r\ge 0} \frac{ r }{1 + r^2} = 2
  $$
  Hence, 
  $M_\GF =2$, $\aopt=-1$, and the canonical hemi-metric $\rhoCanon$
  reproduces
  the parallelogram law,
  $$
  \rhoCanon(x,y) = - \| x+ y\|^2 + (1 - (-1))( \| x\|^2+ \|y\|^2) =  \| x- y\|^2
  .$$
Furthermore,
$\langle  ., . \rangle_2 = \langle  ., . \rangle_H$, i.e.,
the definition of the canonical scalar-product $\langle  .,
. \rangle_2$  is
in accordance with the standard definition
of the scalar product in the pre-Hilbert space.
  \end{example}

\begin{example}[$L_p$ space]
  An extension of the previous example
  to real-valued $L_p$-spaces with
  $p>1$ is not straightforward, even in the finite dimensional case. In fact, defining 
  the entropy as
  $$
  \lnorm x \rnorm = \| x \|^p_p, \qquad x\in\RR^d,
  $$
  gives $M_\GF = 2^{p-1}$. The approach
  leads to unusual
  formulae for both the hemi-scalar product 
   $$
  \langle x, y \rangle_\assoc = \|x + y \|^p - \|x\|^p - \|y\|^p
  $$
  and the hemi-metric
  $$
  \rhoCanon(x,y) = \| x\|^p + \|y\|^p - (2^{p-1} -1)^{-1}(\|x + y \|^p - \|x\|^p - \|y\|^p)
  .$$
  At least, the G\^ateau derivative
  $$
  \langle x, y \rangle_\assoc' := \lim_{t\rightarrow 0} t^{-1}
  \langle tx, y \rangle_\assoc =p \sum_{i=1}^n x_j y_j |y_j|^{p-2}
  $$
  is closely related to the standard semi-inner-product, which is
  $[x,y ] =   p^{-1} \langle x, y \rangle_\assoc'/ \| y \|_p^{p-2}$,
  cf.~\cite{tapia1973characterization} for properties of
  the G\^ateau derivative.
  The discussion is continued in some respects in the next example, as
  the map $x \mapsto \|x\|_p^p$ is a variogram for $p\in (0,2]$.
\end{example}

\begin{example}[Variogram]
A symmetric, real-valued kernel $g$ over a set $G$ is called negative definite
if
   $$
      \sum_{i=1}^n\sum_{j=1}^n a_i a_j g(\nu_i, \nu_j)\le 0
      $$
      for all $n\in\NN$, $\nu_i\in\GF$ and
      $a_i\in\RR$, with $\sum_{i=1}^n a_i =0$.
      Let $\lnorm \cdot \rnorm$ be a
 variogram over  $\GF=\RR^d$, i.e., $\lnorm 0
\rnorm=0$ and the map $(x,y)\mapsto \lnorm x- y \rnorm$ is a negative
definite kernel. 
We assume that $\lnorm \cdot \rnorm\not\equiv 0$ and that $\lnorm
\cdot \rnorm$
 is
 continuous outside the origin.
 We may construct the operator $\circ$ similar to Example \ref{bsp:hilbert}.
If we define $a \plus b = a -b$ for $a,b\in \GF$, then $\signXi = -1$
and $\aopt = 1$. That is,
\begin{eqnarray*}
  \langle x, y \rangle & = &  \langle x, y \rangle_\assoc = \lnorm x \rnorm+\lnorm y \rnorm - \lnorm x - y
                             \rnorm,
  \\
  \rhoCanon(x,y) & = &
                       \lnorm x - y\rnorm
                             .
\end{eqnarray*}
Note that $\langle \cdot, \cdot \rangle$ 
is an ordinary positive definite kernel (Proposition 1.9 in Chapter 1 of
\cite{BCR}), 
and
the Cauchy-Schwarz inequality yields
that $M_\GF  \in [1, 2]$.
\note{

(i) $\sup_{x,y} \gamma(x-y)/ (\gamma(x) + \gamma(y)) \ge  \gamma(x-0)/
(\gamma(x) + \gamma(0))=1$

(ii) $| \gamma(x) + \gamma(y) - \gamma(x-y) | \le 2\sqrt{\gamma(x)
  \gamma(y)}$ So,
$$|1 - \gamma(x-y) / ( \gamma(x) + \gamma(y)) | \le 2\sqrt{\gamma(x)
  \gamma(y) / ( \gamma(x) + \gamma(y))^2 }
= 2\sqrt{\frac{\gamma(x)/\gamma(y)}{ ( 1 + \gamma(y) / \gamma(x))^2}}
  \le 1
  $$
as$x/ (1+x)^2$ has its optimum at $(1,1/4)$.
  Hence, $ -1 \le \gamma(x-y)/ (\gamma(x) + \gamma(y)) \le 2$

  (iii) for $\gamma(x) = \|x\|^\alpha$ we have
  $$\sup_{x,y} \gamma(x - y) / (\gamma(x)
  + \gamma(y)) =  \sup_{e,y}\| e - y\|^\alpha / (1 + \|y\|^\alpha)
  = \sup_r |1+r|^\alpha / (1 + r^\alpha)
  $$
  $M_\GF =1$  happens for many variograms
  }
\end{example}

\begin{example}[Measures]\label{ex:sets}
  Let $\lnorm \cdot \rnorm $ be the Lebesgue measure on the space $\GF $
  of compact sets 
  in $\RR^d$. As $\lnorm A \cup B\rnorm = \lnorm A\rnorm + \lnorm B\rnorm$ for $A\cap
  B=\emptyset$ we may construct a well-defined operator $\circ : G
  \times G \rightarrow G$
  with $\lnorm A \circ B\rnorm = \lnorm A \rnorm +
  \lnorm B\rnorm$ for all compact sets $A$ and $B$ (e.g.\ by suitably
  shifting before taking the union).
  Furthermore, let $\plus$ be $\cup$. Then, we have
  $\signXi = -1$, $m_\GF  = 1/2$, $M_\GF =1$, $\aopt=2$, and
    \begin{eqnarray*}
      \langle A, B\rangle_2& =& \langle A, B\rangle_\assoc =
                              \lnorm A\rnorm + \lnorm B\rnorm - \lnorm
                                A\cup B\rnorm = \lnorm A 
                                \cap B\rnorm,
      \\\rhoCanon(A,B)&=& 2 \lnorm A\cup B\rnorm
                          - \lnorm A\rnorm - \lnorm B\rnorm
                          =\lnorm A\cup B\rnorm - \lnorm A\cap B\rnorm
    \end{eqnarray*}
    That is, $\rhoCanon$ measures the content
    of the symmetric difference $A \Delta B$.

    Replacing the Lebesgue measure by a finite measure $\PP$, all terms
    remain the same, except that the set $A \circ B$  becomes
    ``imaginary'', whenever $\PP(A \circ B)>\PP(\RR^d)$.
    In the context of statistical prediction,
    \cite{makogin2022prediction} consider the excursion metric $E_U$
    between two random variables $Y_1$ and $Y_2$,
    $$
    E_U(Y_1, Y_2) = \EE_U \PP(\{ Y_1 > U \} \Delta \{ Y_2 > U \} \mid U)
    \in[0,1],
    $$
    where $U$ is a fixed random variable independent of $Y_1$ and  $Y_2$.
    That is, the implicit definitions are
    $\lnorm Y_1 \rnorm :=   \PP(\{ Y_1 > U \})$ and
    $Y_1\plus Y_2 := \max\{ Y_1 , Y_2\}$.
    In the context of topological dynamics,
    \cite{abert2012rank} define $\GF$ as the ensemble of
    Borel sets of $X \times \Gamma$, where $X$
    is some Polish space and $\Gamma$ is a finitely generated group
    equipped with the discrete topology. Let
     $\lambda$ be the product measure $\PP \times \nu$, where $\PP$ is some
    probability
    measure and $\nu$ is the counting measure on
    $\Gamma$. They define
    $
    \lnorm M \rnorm = \lambda(M)
    $
     and consider the canonical hemi-metric
    $
    \rhoCanon(M, N) = \lambda(M \Delta N)$ for certain $M,N\in \GF.
    $

\end{example}

\begin{example}[Real axis]\label{ex:realaxis}
  Let $\GF\subset \RR$ be an interval, which includes the value $1$
  and is closed under multiplication. Let
  $\alpha >0$ and $\lnorm  \xi \rnorm = |\xi|^\alpha$.
  Let $c_\alpha = \lnorm 1 \plus 1 \rnorm$ and
  assume that $c_\alpha = \lnorm (-1) \plus (-1) \rnorm$,  if $-1\in \GF $.  
  Then, we have
  $\signXi = +1$ if $c_\alpha >2$ and $\signXi=-1$ if $c_\alpha
  < 2$.
  Let us consider 5 specifications:
  \begin{enumerate}
  \item Let $\GF = \RR$, $\alpha\not=1$, and  $\plus$ be the ordinary
    addition. Then, $c_\alpha = 2^\alpha$, $m_\GF =0$, and  $M_\GF =
    2^{\max\{\alpha-1,0\}}$.
    We distinguish two subcases:
    \begin{description}
    \item[case $\alpha >1$.] Here, $\signXi = +1$, $\aopt = (1 -
      2^{\alpha - 1})^{-1}$,
      \begin{eqnarray*}
        \langle \xi,
        \eta \rangle
        &= & (2^{\alpha - 1}-1)^{-1}
             \left(| \xi + \eta|^\alpha - |\xi|^\alpha - |\eta|^\alpha\right),
        \\
        \rhoCanon(\xi,\eta)
        &= &| \xi|^\alpha + | \eta|^\alpha - 
     \langle \xi,
                             \eta \rangle
                             .
     \end{eqnarray*}
    \item[case $\alpha <1$.] Here, $\signXi = -1$, $\aopt = 1$,
      \begin{eqnarray*}
        \langle \xi,
        \eta \rangle& =& \langle \xi,\eta \rangle_\assoc =  |\xi|^\alpha +
                           |\eta|^\alpha - | \xi + \eta|^\alpha,
        \\\rhoCanon(\xi,\eta)&=& | \xi + \eta|^\alpha
                                 .
      \end{eqnarray*}
      Hence, $\rho(\xi, \xi) > 0$ and $\rho(-\xi ,\xi)=0$ for all $\xi \in
      \RR\setminus\{0\}$.
      This is an undesirable result.
      Since $M_\GF =1$, the hemi-metric $\rho_\infty =  \langle \cdot , \cdot
      \rangle$ is an appealing alternative.
    \end{description}
  \item We reconsider the previous example, but let $\GF = [0,\infty)$.
    While the case $\alpha > 1$ remains unchanged we get for $\alpha
    <1$ that $m_\GF  = 2^{\alpha -1}$. Hence,
    $$
    \rhoCanon(\xi,\eta)=  \xi^\alpha +  \eta^\alpha -
    \left(1 -2^{\alpha-1}\right)^{-1}  \langle \xi,
    \eta \rangle_\assoc
    .
    $$
 \item \label{ex:frechet.1}\label{ex:max.real}
   Let $\GF =
        [0,\infty)$, $\alpha>0$, and  $\plus$ be the ordinary
    maximum denoted by $\vee$.
    Then, 
    $\signXi = -1$, $m_\GF  = 1/2$, $\aopt = 2$, and
    \begin{eqnarray*}
      \langle \xi,
      \eta \rangle_2& =& \langle \xi,\eta \rangle_\assoc   =  \xi^\alpha +
                       \eta^\alpha - ( \xi \vee \eta)^\alpha
                         = ( \xi \wedge \eta)^\alpha,
      \\\rhoCanon(\xi,\eta)&=& 
                               ( \xi \vee \eta)^\alpha -( \xi \wedge \eta)^\alpha
                               .
    \end{eqnarray*}
    \cite{stoevtaqqu05} show that $\rhoCanon$ is an ordinary
    metric. 
    \cite{rottger2023total} deal with the total positivity of order 2, which
    translates into $\rhoCanon(\xi,\eta)\ge0$.
    
  \item 
      We reconsider the previous example, but let
      $\GF = \RR$.  Then, $m_\GF  = 0$ and
    \begin{eqnarray*}
      \langle \xi,
      \eta \rangle& =& \langle \xi,\eta \rangle_\assoc 
                       = | \xi \wedge \eta|^\alpha,
      \\\rhoCanon(\xi,\eta)&=& | \xi \vee \eta|^\alpha
                               .
    \end{eqnarray*}
    Note the asymmetry between the positive and the negative real axis:
    $\rho(\xi, 0) = \xi^\alpha$ and $\rho(-\xi,0)=0$ for all $\xi >0$.
    The limit hemi-metric  $\rho_\infty = \langle \cdot,
    \cdot \rangle $ is not a good alternative.
    A kind of symmetry is obtained if the negative axis is supplied with the
    minimum. If $x \plus y := x +  (y-x) \1_{|y| > |x|}$,
    then we get 
    for all $\xi,\eta\in\RR$ that
    \begin{eqnarray*}
      \langle \xi,
      \eta \rangle& =& \langle \xi,\eta \rangle_\assoc 
                       = ( |\xi| \wedge |\eta|)^\alpha,
      \\
      \rhoCanon(\xi,\eta)
                  &=&
                      ( |\xi| \vee |\eta|)^\alpha - ( |\xi| \wedge
                      |\eta|)^\alpha
                      .
    \end{eqnarray*}
    Note, that the positive and negative half axes now
    parallel and point into the same direction. 
  \item \label{ex:cauchy}
    If $c_\alpha =2$, then the $\signXi$ is undetermined. For
    instance, consider the family of Cauchy distributions
    ${\rm Cauchy}(\lambda)$, $\lambda\in \RR$, with densities 
    $$f_\lambda(x) =  \frac1{|\lambda| \pi } \cdot \frac1{1 +
      (x/|\lambda|)^2},
    \qquad x\in\RR, \lambda \not=0
    .
    $$ 
    Then, $\lambda \mapsto
    | \lambda|$ defines an entropy on $\RR$
    that models the addition of independent Cauchy variables.
    All further discussions are identical to the second case of the
    first example
    with $\alpha=1$ instead of $\alpha<1$.


  \end{enumerate}
   We learn from the above specifications, that the choice of the
   parameter space $\GF $ is crucial.
   See Section \ref{sec:trafo} for a general treatment of the situations
   discussed in this example.
   See also the discussion on p.~1246 in
   \cite{mccullagh2002statistical}, where  $\RR$ is considered as
   the parameter space for the scale and which is afterwards
   restricted to $[0,\infty)$ for the parameter identification.

 \end{example}


\begin{example}[Mutual information]\label{ex:mutual}
  Let $\lnorm \cdot \rnorm$ be the Shannon entropy.
  The mutual information is
  defined  as
  \begin{eqnarray*}
    I(X, Y) &=&
           \lnorm p_X \rnorm + \lnorm p_Y \rnorm - \lnorm p_{XY} \rnorm,
   \end{eqnarray*}
 where $p_{XY}$ is the joint probability of the discrete random
  variables $X$ and $Y$ and $p_{X}$ and $p_{Y}$ are the respective margins.
  We define the operator $\plus$ as
  $p_X \plus  p_Y = p_{XY}$, assuming that $p_{XY}$ is unique
  for all $p_X$ and $p_Y$ in the given setup. We further assume that
  $p_{XX}(z,z)=p_X(z)$.
  Then,  we have $\signXi = -1$, $m_\GF  = 1/2$, $\aopt=2$,
  and the mutual
  information equals the canonical hemi-scalar product $\langle \cdot,
  \cdot \rangle_2 \equiv \langle \cdot,
  \cdot \rangle_\assoc$.
  \note{
  $m_\GF  \le 1/2$ as $\lnorm p_{XX} \rnorm / (2 \lnorm p_{X} \rnorm )
  =1/2$
  Further, $m_\GF  \ge 1/2$ as  $$\lnorm p_{XY} \rnorm \ge \max\{ \lnorm
  p_{X} \rnorm, \lnorm p_{Y} \rnorm\} / 
  (\lnorm  p_{X} \rnorm + \lnorm p_{Y} \rnorm) \ge \inf _{y>0}
  \max\{1,y\}/(1+y)  \ge 1/2
  $$
  }
  The canonical hemi-metric 
  \begin{eqnarray*}
    \rhoCanon(p_X, p_Y)
    = 2  \lnorm p_{XY} \rnorm -  \lnorm p_{X}
        \rnorm -  \lnorm p_{Y} \rnorm
  \end{eqnarray*}
  is indeed a metric, called variation of information
  \citep{meilua2007comparing}.
  Since $M_\GF =1$, the mutual information $I$ might also be considered as a 
   hemi-metric (Remark \ref{bem:infty}).

 \end{example}

\begin{example}[Poisson distribution]
  We construct a bivariate Poisson distributed random vector as follows.
  Let $\lambda,\mu>0$, $a,b\in[0,1]$, $\nu = \min\{a\lambda, b\mu\}$
  and let
  \begin{eqnarray*}
   X_0 &\sim &\Pois(\lambda -\nu)\\
   Y_0 &\sim &\Pois(\mu -\nu)\\
   Z & \sim & \Pois(\nu)
  \end{eqnarray*}
  be independent variables. Let $X = X_0 + Z$ and $Y = Y_0 + Z$. We
  consider $(X, Y)$. If $\lnorm \cdot \rnorm$ is the Shannon entropy,
  numerical calculations suggest that $\signXi=-1$ for all $a,b\in
  [0,1]$, so that the mutual information is the associated hemi-scalar
  product. 

  Now, let the entropy  $\lnorm \cdot \rnorm$ be defined as $\lambda$
  for a univariate Poisson random variable with expectation $\lambda$.
  A natural extension
  of this entropy to our bivariate
  Poisson random vector is the sum of the entropies of the included
   independent variables $X_0$, $Y_0$, and $Z$, i.e.
  $$
  \lnorm (X, Y) \rnorm := \lnorm X \plus Y \rnorm :=  \lnorm X_0 \rnorm + \lnorm Y_0 \rnorm +
  \lnorm Z \rnorm
  = \lambda + \mu - \nu
  , 
  $$ 
  so that  $\signXi = -1$ 
  for all $a,b\in[0,1]$ and $\lambda, \mu>0$.
  Thus, the associated
  hemi-scalar product equals
  $$
  \langle X, Y \rangle_\assoc = \nu = \min\{a\lambda, b\mu\}
  .
  $$
  Since
  $m_\GF  = 1 - \min\{a,b\}/2$, we have
  $\aopt = 2 / \min\{a,b\}$ for $ab>0$.
  \note{
  $$
  \min_\lambda\frac{2\lambda - \lambda \min\{a,b\}}{2 \lambda} = 1 -
  \min\{a/b\}/2 
  $$
  UND NICHT:
  $m_\GF  = {ab}/({a + b})$, we have
  $\aopt = (a+b) / (a+b -ab)$.
  da
  $$
  \min \frac{\lambda + \mu -\nu}{\lambda + \mu}
  = 1 - \max_{\lambda,\mu} \frac{\min\{a\lambda, b\mu\}}{\lambda + \mu}
  = 1 - \max_\mu  \frac{\min\{a, b\mu\}}{1 + \mu}
  = \frac{ab}{a + b}
  $$

  }

\end{example}

\begin{example}[Shannon entropy for dependable systems]
  \label{ex:shannon.dependable}

  We continue Example \ref{ex:shannon} and 
  let $q_x\in [0,1]$, $x\in \alpha$,
  which might be interpreted
  as the probability that a letter is trustably transferred
  or that a model value or a measured value is reliable.
  If a value is not dependable, then it  is seen as being
  drawn from some other distribution. We
  are not interested in the other distribution, nor whether there is a
  random coincidence between the true value  and the value from the
  other distribution.
  We do not follow the lines of the noisy-channel coding theorem, but
  follow the definition of the cross-entropy and define
  \begin{eqnarray}
    \label{eq:q}
  \lnorm (p, q) \rnorm_u = -u(p,q) \sum_x q_x p_x \log p_x,
  \quad
  p_x, q_x\in[0,1], \sum_x p_x =1   
  \end{eqnarray}
  for some $u(p,q)>0$.
  A pair of values shall be dependable if both components are dependable.
 For $q\not \equiv 0$ and $\tilde q \not \equiv 0$ 
 we define
 $(\tilde p,\tilde q) \circ ( p, q )$ as the cross-product
 $\tilde p \times  p, \tilde q \times q$,
  so that we get
  for independent values $(p,q)$ and $(\tilde p, \tilde q)$
  with $q\not \equiv0$ and $\tilde q \not \equiv 0$ 
  \begin{eqnarray*}
   \frac{ \lnorm ( \tilde p,\tilde q) \circ ( p, q )) \rnorm_u}{u(\tilde p \times  p, \tilde q \times q)}
    = 
         -         
         \sum _{x,y} (\tilde q_yq_x ) \tilde p_y p_x \log(\tilde p_yp_x )
    =\frac{ 
           \sum _{y}  \tilde q_y \tilde p_y}{ u(p,q)}\lnorm (p, q) \rnorm_u
           +
           \frac{ 
           \sum _{x} q_x p_x }{u(\tilde p,\tilde q)}
           \lnorm (\tilde p, \tilde q) \rnorm_u
  \end{eqnarray*}
  
  Hence,
  \begin{eqnarray*}
    \lnorm ( \tilde p,\tilde q) \circ ( p, q )) \rnorm_u
         & =& \lnorm (\tilde p, \tilde q) \rnorm_u + \lnorm (p, q) \rnorm_u
   \end{eqnarray*}
  
  for arbitrary $(p,q)$ and $(\tilde p, \tilde q)$
  if and only if $u(p,q) =  1 / \sum _{x} q_x p_x  $
  for  $\sum _{x} q_x p_x\not=0$. 
  In particular, 
  $\lnorm (p,q) \rnorm =\lnorm (p,aq) \rnorm$
  for all $a>0$, that is,
  $\lnorm . \rnorm$ attributes entropy only
  to the shape of $(q_x)_x$
  and $q$ may take arbitrary non-negative values.
  Note that  $r_x := q_x p_x /  \sum _{y} q_y p_y$
  can be interpreted as the conditional probability of a value given
  it is dependable.

  An appropriate definition of the binary operator $\plus$
  is not clear cut. We consider the following approach.
   Let
    \begin{eqnarray}\label{eq:KLplus}
     (p, q)  \plus  (\tilde p, \tilde q )
      & =&
           \left(p, q + \tilde q \frac {\tilde p} p \right)
           ,
    \end{eqnarray}
    i.e., the operator $\plus$  allows an
    interpretation of updating model $p$
    with reliability $q$ by data $\tilde p$
    with reliability $\tilde q$.
    We get
    $$
    \lnorm (p, q)  \plus  (\tilde p, \tilde q  ) \rnorm_u
    =\frac{-1}{\sum_x ( q_x p_x +  \tilde q_x \tilde p_x )}
    \sum_x ( q_x p_x +  \tilde q_x \tilde p_x ) \log p_x
    $$
    In particular,
    \begin{eqnarray*}
      \label{eq:1}
   \lnorm (p, q)  \plus  (p, q  ) \rnorm_u
    =  \lnorm (p, q)\rnorm_u
    ,     
    \end{eqnarray*}
     which underlines the above interpretation as twice the same
    information with the same uncertainty does not give any additional information.
    Now, $\signXi=-1$, $m_\GF  = 0$, $\aopt=1$,
     \begin{eqnarray*}
        \langle  (p, q), (\tilde p, \tilde q)
       \rangle_u& =& \lnorm (p, q)  
                     \rnorm_u
                     +\lnorm  (\tilde p, \tilde q  ) \rnorm_u
                     - \lnorm (p, q)  \plus  (\tilde p, \tilde q  )
                     \rnorm _u,
       \\
       \rhoCanon^{u}((p, q), (\tilde p, \tilde q) )
              &=& \lnorm (p, q)  \plus  (\tilde p, \tilde q  )
                  \rnorm _u
                  .
     \end{eqnarray*}                                          
     Let $\bf 0$ and $\bf 1$ be sufficiently long
     vectors of $0$'s and $1$'s, respectively. Then, we get as a
     special case that 
     \begin{eqnarray}\label{eq:KL.ansatz}
        \langle  (p, {\bf 0}), (\tilde p, {\bf 1})
       \rangle_u& =&   \sum_x\tilde  p_x \log  \frac{p_x}{ \tilde p_x}
                     = -d_{KL}(\tilde  p, p)
                     ,
     \end{eqnarray}              
     where $d_{KL}$ denotes the  Kullback-Leibler divergence.
     \label{ex:mle}
     The maximum likelihood approach for independent and identically
     distributed
     observations
     \citep{akaike73} 
   minimizes
   $$-\sum_x \tilde p_x \log
   p_x^{(\theta)} = \rhoCanon^u( (p^{(\theta)}, {\bf 0}), (\tilde p, {\bf 1}))
   $$
   where $p_x^{(\theta)}$ is the
   probability of observation $x$ in model $\theta$, and $\tilde p_x$
   equals the empirical frequency. Hence, the above ansatz allows a Bayesian
   interpretation of the maximum likelihood, namely that we have save data $(\tilde p, {\bf 1})$
   and models $(p^{(\theta)}, {\bf 0})$ with maximal uncertainty.

      \end{example}

  \begin{example}[Kullback-Leibler divergence]
    \label{ex:leibler}\label{ex:KL}
    While many other important quantities, such as the
    symmetric difference (Example \ref{ex:sets}), the mutual
    information and the
    variation of information (Example \ref{ex:mutual}) appear
    naturally in our approach, a satisfying approach to the
    Kullback-Leibler divergence has not been found yet.
    Lack of known alternatives, we refine Equation
    \eqref{eq:KL.ansatz}.
    The latter
    suggests to restrict the space
    of $(p,q)$ to pairs of the form $(p, {\bf 0})$ on the left hand side of
    \eqref{eq:KLplus}  and to pairs of the form $(p,
    {\bf 1})$
    on the right hand side, so that
    the map \eqref{eq:cond2.comparable} must be replaced by
    $$
    \xi \mapsto  \lnorm (\xi, {\bf0})  \plus  (\xi, {\bf1}  )\rnorm_u
    -  \lnorm (\xi, {\bf0}) \rnorm_u - \lnorm (\xi, {\bf1}  )\rnorm_u
    .$$
    This function is identically $0$, so that its sign is
    undetermined. We define it as $+1$, so that
  \begin{eqnarray*}
    \langle p, \tilde p\rangle_\assoc
    &:= & -\langle  (p, {\bf 0}), (\tilde p,
          {\bf1})\rangle_u = d_{KL}(\tilde p,   p)
          .
    \end{eqnarray*}     
   As $m_\GF =1$, the Kullback-Leibler divergence equals also the
  hemi-metric $\rho_\infty$.
  \end{example}

  \begin{example}[Tsallis entropy]\label{ex:tsallis}
    The Tsallis entropy
    \citep{tsallis88}
    generalizes the Shannon entropy and 
    is defined  for a discrete probability distribution $p$
    as
    $$
    \lnorm p \rnorm = \frac{k}{q-1} \left(1- \sum_i p_i^q\right)
    .$$
    Here, $q\in \RR$ and $k>0$. 
    We define $p \plus p'$ as $p
    \times p'$, so that
    $$
    \lnorm p \plus p' \rnorm = 
    \lnorm p \rnorm +
    \lnorm  p' \rnorm - \frac{q-1}k
    \lnorm p  \rnorm \lnorm p' \rnorm 
    .$$    
    Hence, the common point of view with \cite{tsallis88} is that we
    both start with independent systems. But we assume  that they interact,
    when they are joined.
    The same considerations apply also to the Sharma-Mittal entropy,
    $$
    \lnorm p \rnorm = \frac{k}{q-1} \left(1- \left(\sum_i p_i^q\right)^{1/k}\right)
    .$$
    Here, $q,k\in\RR$ \citep{sharma1975new}.

  \end{example}

\begin{example}[Tichonov regularization]\label{ex:tichonov}
  Consider the orthogonal design case of a
  linear model $Y = X  \beta + \varepsilon$ \citep{tibshirani96}.
  That is, $X \in \RR^{n\times p}$, $\rank X = p$,
  $X^\top X$ is the identity,
  $\beta \in\RR^p$, and $\varepsilon \in
  \RR^n$ is a vector of errors.
  Let $y$ be a realization of $Y$.
 For
  $\lambda>0$, the Tichonov regularization
  $$
  \| y - X \beta \|^2 + \lambda \| \beta \|^2 = \min_\beta !
  $$
  is equivalent to
  $$
  \frac{-1}{1 + \lambda} \|y + X \beta \|^2 +
  \left( 1- \frac{-1 }{1 + \lambda}\right) (\| y\|^2 +  \| X \beta \|^2)
  =\min_\beta !
  .$$
  Hence, in the general setting,
  we may interpret the minimization of the hemi-metric
  $q_a$, $a \in \Xi \cap(-\infty,0)$, as a generalized Tichonov regularization.
  Or reversely, the Tichonov regularization can be interpreted as a
  non-canonical
  choice of $a\in\Xi$ within  a pre-Hilbert space framework,
  cf.\ Example \ref{bsp:hilbert}.
  Note that, in contrast to the maximum likelihood principle in
  Example \ref{ex:mle}, a Bayesian interpretation of the
  general Tichonov regularization does not seem to be straightforward
  from our entropy point of view given here.
\end{example}

\begin{example}[Bivariate Gaussian]
  Let $X$ be a standard Gaussian random variable
  and the entropy of a vector be the sum of the entropies of the components.
Then, the random vectors  $(1, -1)^\top X$ and $(1,1)^\top X$ are
jointly multivariate Gaussian, fully dependent, but uncorrelated
according to Definition \ref{def:scalarproduct}.
Note that the standard notion of uncorrelatedness is defined only for
scalar random variables. The generalized definition still implies that
two jointly bivariate normal, scalar Gaussian random variables are uncorrelated if
and only if they are independent.
\end{example}

\note{

\begin{example}[Thermodynamic formalism]
  The thermodynamic formalism is given by
  \begin{eqnarray}\label{eq:TF}    
  S(\sigma) - \sigma(U) = \max_\sigma   !
  \end{eqnarray}
   where $\sigma$ denotes a discrete probability distribution,
  $$
  S(\sigma) = -\sum_x \sigma_x \log \sigma_x
  $$
  and
  $$
  \sigma(U) = -\sum_x \sigma_x \log e^{-U(x)}
  $$
  for a given energy function $U$.
  cf.\ \cite[p.~3]{ruelle}.
  Hence, \eqref{eq:TF} is identical to minimizing
  the Kullbach-Leibler divergence.
  \footnote{Check!}

  NOTE: ``dependable'' construction is nothing else than rejection
  sampling, where rejected is marked as ``non-dependable''.
  
\end{example}

\begin{example}[Permutation group]
  ??? Fehlstandszahl/Inversionszahl als Entropie-Mass??
  Fehlstandszahl hat eine nette Rekursionsformel.
\end{example}
}

 \begin{example}[Tropical scalar product]
   For a suitable linear space $V$  of non-negative functions, let
   $\lnorm f \rnorm = \sup f $ and $\plus$ be the usual addition.
   Then $m_\GF=1/2$, $M_\GF=1$, and $\signXi$ is undefined. We set $\signXi=+1$.
   Then, the associated scalar product
   $$
   \langle f, g \rangle_\assoc =\sup (f + g)-   \sup f - \sup g 
   $$
   differs substantially from the $(\min, +)$-product \citep{maslov1987new},
     which is
   $$
   \langle f, g \rangle_{\min,+} =
   \inf (f + g)
   .$$
   Now, let $\lnorm f \rnorm = \inf f$. Then, $\Xi = [0,\infty)$ and
   the tropical scalar product
   $
   \langle f, g \rangle_{\min,+} = \rho_1(f, g)
   $
   is regarded here as a hemi-metric.
 \end{example}

\subsection{Properties}

The next proposition shows what remains from the known properties of
a metric and a scalar product in the framework of a hemi-metric $\rho_a$ and
a hemi-scalar product $\langle \cdot, \cdot \rangle$. By definition,
$\rho$ is non-negative and Equation \eqref{eq:scalar.rho} shows that
the entropy can also be interpreted as the entropy difference to 
a deterministic system.
The two familiar properties that the triangular inequality holds
and that
the distance to itself is $0$ are completely lost. Proposition
\ref{prop:scalar2} shows both that under additional assumptions the
latter property can be regained and that this property is exceptional.
For the triangular inequality no analogue has been found yet in this
general setup.  Equation \eqref{eq:rho.scaling} in Proposition
\ref{prop:scalar2}  shows that a weak scaling property of the
hemi-metric can be recovered under additional assumptions.

The definition of the hemi-scalar product implies that $\langle \xi,
\xi \rangle$ is always non-negative, cf.~Equation \eqref{eq:sp.nonneg}, and
that the hemi-scalar product is always $0$ whenever a deterministic
system is involved. With additional assumptions, we get that the
hemi-scalar product with itself is $0$ only for the deterministic
systems, cf.~Equation \eqref{eq:scalar.Gs}.
However, linearity in the common sense or even
bilinearity seem to be a luxury asset. The weak form of linearity
manifested in Equation \eqref{eq:sp.linear} is sufficient to gain
important results in Section~\ref{sec:constructions}. Under additional
assumptions, a further partial linear property, cf.~Equation \eqref{eq:scalar.scaling},
might be gained.
Two other, general properties of the hemi-scalar product are
surprisingly nice. First, the hemi-scalar product
can always be recovered from the canonical hemi-metric, cf.~Equation 
\eqref{eq:sp.recover}. (The reverse is unfortunately not true 
as Inequality \eqref{eq:xiGmxi} shows.)
Second, (one-sided) bounds for $\langle \xi, \eta\rangle / \lnorm \xi \circ \eta
\rnorm$  always exist, cf.~Equations \eqref{eq:sp.CSa} and \eqref{eq:sp.CS2}.

\begin{proposition}\label{lemma:scalar}
  Let $(\GF,\circ,  \plus, \lnorm .\rnorm)$ be a
  entropy-driven  hemi-group with comparable elements.
  Let $a \in \Xi \setminus\{0\}$ and denote by
  $\langle \cdot, \cdot\rangle_*$ any of $\langle \cdot, \cdot\rangle$, $\langle
  \cdot, \cdot\rangle_2$, or $\langle \cdot, \cdot\rangle_\assoc$.
  Then, for all  $\xi,\eta,\nu \in \GF$ and $ \varepsilon \in \detsys$, we have
  \begin{eqnarray}
    \nonumber
    \langle \xi , \eta \rangle_*
    &= & \langle \eta, \xi \rangle_*
         , \mbox{ iff }
         \lnorm \xi \plus \eta \rnorm =  \lnorm \eta \plus \xi \rnorm
     ,\\
   \label{eq:sp.nonneg}
    \langle  \xi , \xi  \rangle_* &\ge& 0
 ,\\ \label{eq:sp.linear}
      \langle  \xi \plus \eta , \nu  \rangle_*  &=
       &
          \langle  \xi , \eta \plus \nu  \rangle_* - \langle  \xi, \eta
      \rangle_* +
         \langle  \eta, \nu \rangle_*
    ,\\ \label{eq:R0=0} 
    \langle \xi , \varepsilon \rangle_*
    &=&0 
     %
    %
    %
             %
           ,\\\label{eq:sp.CS0a}
  \sign \Xi    \langle \xi, \eta \rangle_\assoc & \ge & -  \frac{\lnorm
                                              \xi \circ \eta
                                              \rnorm}{\sup\Xi}
         ,\\\label{eq:sp.CS0b}
   \sign \Xi  \langle \xi, \eta \rangle_\assoc
    & \le &- \frac{ \lnorm
            \xi \circ \eta \rnorm}{\inf \Xi},
              \quad \mbox{if }\inf \Xi \not = 0
           ,\\\label{eq:sp.CS0aa}
    \sign \Xi  \langle \xi, \eta \rangle_\assoc & \ge &  \frac{\lnorm
                                              \xi \plus \eta
                                              \rnorm}{1-\sup\Xi}
                                 \quad \mbox{if }    \sup  \Xi \not = 1
             ,\\\label{eq:sp.CS0c}
    \sign \Xi   \langle \xi, \eta \rangle_\assoc & \le &  \frac{\lnorm
            \xi \plus \eta \rnorm}{1- \inf \Xi}
 ,\\ \nonumber
    \rho_a( \xi , \eta )  &= & 
        \rho_a(\eta, \xi )
         , \mbox{ iff }
         \lnorm \xi \plus \eta \rnorm =  \lnorm \eta \plus \xi \rnorm         
   ,\\\label{eq:scalar.rho}
    \rho_a(\xi, \varepsilon)  &= & 
                            \lnorm \xi
                           \rnorm
          .
  \end{eqnarray}
\end{proposition}
\begin{proof}
  We show only Equations \eqref{eq:sp.CS0a}--\eqref{eq:sp.CS0c};
  all other assertions follow immediately from the definitions.
  We have for all $a\in \Xi$ and $\xi,\eta\in \GF$ that
  $$
  0\le\lnorm \xi \circ \eta \rnorm +  a (\lnorm \xi \plus \eta \rnorm
  - \lnorm \xi \circ \eta \rnorm)
  =  \lnorm \xi \plus \eta \rnorm - (1-a)( \lnorm \xi \plus \eta \rnorm - \lnorm \xi \circ \eta \rnorm)
  $$
  i.e., for all $a\in \Xi$, we have
  \begin{eqnarray*}
    - a  \signXi \langle \xi , \eta \rangle_\assoc &\le&\lnorm \xi \circ \eta \rnorm,
    \\
    (1-a) \signXi  \langle \xi , \eta \rangle_\assoc &\le& \lnorm \xi \plus
                                                           \eta \rnorm
                                                           .
  \end{eqnarray*}
  Inequalities \eqref{eq:sp.CS0a}--\eqref{eq:sp.CS0c} follow
  immediately. Note that equality in \eqref{eq:sp.CS0a} and \eqref{eq:sp.CS0aa}
  is reached
  if
  and only if $\lnorm \xi \plus \eta \rnorm = m_G \lnorm \xi \circ
  \eta \rnorm$, and in \eqref{eq:sp.CS0b} and \eqref{eq:sp.CS0c} if
  and only if $\lnorm \xi \plus \eta \rnorm = M_G \lnorm \xi \circ
  \eta \rnorm$.
\endproof\end{proof}

\begin{proposition}\label{cor:scalar}
  Let $(\GF,\circ,  \plus, \lnorm .\rnorm)$ be a
  entropy-driven  hemi-group with comparable elements.
  Assume that $\aopt$ is finite.
  Then, for all  $\xi,\eta \in \GF$, we have
  \begin{eqnarray}
    \label{eq:sp.CSa}
    \langle \xi, \eta \rangle
    & \le &  \lnorm
            \xi \circ \eta \rnorm
         ,\\\label{eq:sp.CSb}
    \langle \xi, \eta \rangle
    & \le &
   \frac{\aopt}{ \aopt-1}\lnorm \xi \plus \eta \rnorm,
            \quad \mbox{if } \aopt\not\in [0,1]
             ,\\\label{eq:sp.CS2}
     \langle \xi, \eta \rangle
    & \ge & \frac{\aopt}{\sup \Xi}\lnorm
            \xi \circ \eta \rnorm,
             \quad \mbox{if } \aopt < 0
     ,\\\label{eq:sp.recover}
    \langle \xi, \eta \rangle
    &=& \rhoCanon(\xi, \varepsilon) +  \rhoCanon( \eta, \varepsilon) - \rhoCanon(\xi, \eta)
   \end{eqnarray}

\end{proposition}

\begin{proof}
   Inequalities  \eqref{eq:sp.CSa}--\eqref{eq:sp.CS2} can be derived
   from Inequalities \eqref{eq:sp.CS0a}--\eqref{eq:sp.CS0c} using the
   fact that  $\signXi\aopt  \le0$.
   They can be derived also in an elementary way as follows.  Inequality \eqref{eq:sp.CSa}
  is immediate from the definition.
  Since $(a-1)/a\in(0,\infty)$ for $a \not \in [0,1]$,
  Inequality \eqref{eq:sp.CSb} follows from \eqref{eq:rhoopt} and
  \eqref{eq:rhoopt2}.   
   Inequality \eqref{eq:sp.CS2} is equivalent to
   \begin{eqnarray}
     \label{eq:sp.CS2a}
     0\le 
     \lnorm \xi \plus \eta \rnorm  + \left(
     \frac{1}{\sup \Xi} -1 \right)  \lnorm \xi \circ \eta \rnorm
     .
   \end{eqnarray}
  If $\sup \Xi < \infty$, then \eqref{eq:sp.CS2a} is equivalent to
  $
   0\le \sup \Xi \lnorm \xi \plus \eta \rnorm  + \left(1 - \sup
     \Xi\right)  \lnorm \xi \circ \eta \rnorm 
   .$
   If  $\sup \Xi = \infty$, then \eqref{eq:sp.CS2a} is equivalent to
   \begin{eqnarray}
     \label{eq:sp.CS2b}
     0 \le  \lnorm \xi \plus \eta \rnorm  - \lnorm \xi \circ \eta \rnorm
     . 
   \end{eqnarray}
   If $\lnorm \xi \circ \eta \rnorm=0$, then  $\lnorm \xi \plus \eta
   \rnorm=0$.
   If $\lnorm \xi \circ \eta \rnorm\not= 0$, then \eqref{eq:sp.CS2b}
   is implied by $\sup \Xi = \infty$.
   Equality \eqref{eq:sp.recover} follows immediately from the definitions.
\endproof\end{proof}

\begin{remark}\label{rem:r}
   Inequality \eqref{eq:sp.CSa} implies that
  $$
  r_{\xi,\eta} :=
  \frac{\langle \xi, \eta \rangle}{\lnorm \xi \circ \eta \rnorm} \le 1
  .$$
  If $ -\sup \Xi \le \aopt < 0$, i.e., $2 \le m_\GF + M_\GF$,
  then  Inequality \eqref{eq:sp.CS2} implies
  $$
  -1 \le  \frac{\aopt }{\sup \Xi}\le r_{\xi,\eta} ,
  $$
  so that $r_{\xi,\eta}$ can be interpreted as a
  correlation coefficient.
  Clearly, if
  $\lnorm \xi \rnorm = \lnorm \eta
  \rnorm$, then
  $r_{\xi,\eta}
  $
  coincides with the Pearson correlation
  coefficient,
  \begin{eqnarray}
    \label{eq:pearson}
    \rho_{\xi, \eta} :=  \frac{\langle \xi, \eta \rangle_2}{ \sqrt{\lnorm \xi \rnorm  \lnorm
      \eta \rnorm}}
  .
  \end{eqnarray}
  The correlation coefficient $r_{\xi,\eta}$ is also closely related to 
  the coefficient of
  determination $R ^2$ of a linear model
  $Y = X\beta + \varepsilon$.
  Here, $X \in \RR^{n\times p}$, $\rank X = p$, the first column of $X$
  is identically~$1$, $\beta \in\RR^p$, and $\varepsilon \in
  \RR^n$ is a vector of errors.
  Let $H_\emptyset=0\in\RR^{n\times n}$ and for
  $A\subset\{1,\ldots, p\}$, $A\not = \emptyset$, let
  $H_A = X_A(X_A  ^\top X_A)^{-1} X_A^\top$, where $X_A$ consists of the
  columns of $X$
  given by $A$. Define $\lnorm A \rnorm = \| y - H_Ay\|^2_2$
  for a realization $y$ of $Y$ and let $A\plus B =
  A \cap B $.
  Then, $\sign \Xi = -1$, $m_\GF = 1/2$  
  and $a = 2$,
  so that $\langle A, B \rangle = 2 \lnorm A\rnorm$ for $B \subset A$.
  We let
  $$
  R^2_{A,B} := 1-\frac{\lnorm A\rnorm}{\lnorm B\rnorm}
   =  1 -\frac{r_{A,B} }{2 - r_{A,B}}
   \qquad \mbox{ for } B\subset A
   \mbox{ and} \lnorm B \rnorm \not = 0,
   $$
   hence
   $$
   r_{A,B} = 1 - \frac{R^2_{A,B}}{2- R^2_{A,B}}\in[0, 1]
   .
   $$
   The coefficient of
   determination $R ^2$
   equals $R^2_{A,\{1\}}$
   for a submodel  $Y = X_A\beta_A + \varepsilon$
   with $A \supset \{ 1 \}$.
\end{remark}
 \note{{
  \color{red}
  Note $\lnorm A \cap B \rnorm \ge
  \max\{ \lnorm A \rnorm, \lnorm B \rnorm\}$.
  Hence,
  $$\lnorm A \cap B \rnorm/ \lnorm A \circ B \rnorm \ge \frac1{1 +
    \frac {\min\{ \lnorm A \rnorm, \lnorm B \rnorm\}}{ \max\{ \lnorm A
      \rnorm, \lnorm B \rnorm\}}}
  \ge 1/2
  $$
  A value for $M_\GF=M_\GF(y)$ is difficult to determine?!

  Alternatively,
  $A\plus B =
  A \cup B $. Then $m_\GF = 0$, $a=1$, $ \langle A, B \rangle = \lnorm
  B\rnorm$
  and
  $$R^2 = 2 - \frac1{r_{A,B}}
  $$
  or
  $$
  r_{A,B} = (2- R^2)^{-1}
  $$
  mit $ r_{AB} \in [1/2, 1]$.
  Schoenere Formel mit Vorteil, dass $R^2$ und $r$ gleichlaufend sind,
  aber mit dem Nachteil, dass $r$ nicht als Korrelation
  interpretierbar ist.

  The alternative definition $A\plus B =
   A \cup B $ would lead to $ r_{A,B} = (2- R^2_{A,B})^{-1}\in[1/2,1]$ with
   the advantage that $R^2_{A,B}=1$ if and only if $r_{A,B}=1$, but
   also with the loss of interpretability of $r_{A,B}$ as a
   correlation coefficient.

   While coefficient of
   determination determine
   the correlation coefficient
   $r_{A,B}$ can be seen as a generalization of $R^2$, which is able to
   compare any two submodels.
   }
  }

\begin{definition}
  An   entropy-driven  hemi-group with comparable elements
  is said to have hemi-associative elements, if 
  \begin{eqnarray}
\label{eq:AG}
  \lnorm (\xi \plus \eta) \plus \nu \rnorm
                               &= &\lnorm \xi \plus (\eta \plus \nu)
  \rnorm,\qquad \forall  \xi ,\eta,\nu \in \GF.
  \end{eqnarray}
  We write
  $n \lambda = \lambda \plus\ldots \plus \lambda$ ($n$ times) for
 $n\in\NN$ and $\lambda \in \GF$,  and let $0\lambda \in \detsys$.
\end{definition}

For an  entropy-driven  hemi-group with hemi-associative, comparable
elements
we have a partial symmetry of the hemi-scalar product, namely,
  \begin{eqnarray} \label{eq:sp.symm}
    \langle  i\xi, j \xi  \rangle_*
    &=&  \langle  j\xi, i \xi \rangle_*
        \qquad \forall \xi\in\GF, i,j\in\NN_0
        .
  \end{eqnarray}
  Here, $*$ signifies the associated hemi-scalar product or a
  canonical hemi-scalar product.
  
\begin{proposition}[Generalization of the Cauchy-Schwarz inequality]\label{prop:CS}
  Let $(\GF,\circ,  \plus, \lnorm .\rnorm)$ be a
  entropy-driven  hemi-group with hemi-associative, comparable
  elements, and
  \begin{eqnarray*}
    c_m  &:=& \sup_{\xi \in \GF,\setminus \GF_0}
              \frac{ \lnorm m \xi \rnorm}{ \lnorm \xi \rnorm} < \infty
              \quad \forall m \in \NN,
    \\
     S_+ & :=& \inf_{\footnotesize
               \begin{array}{c}
                 \xi, \eta \in \GF\\
                 \signXi \langle \xi,  \eta\rangle_\assoc < 0\\
                  \end{array}
                 }
                 \liminf_{
                 m,n\rightarrow\infty}
                \frac{\langle m \xi, n \eta\rangle_\assoc}{\sqrt{c_m c_n}
    \langle \xi,  \eta\rangle_\assoc} >0
    .
   \end{eqnarray*}
 If, for any $c>0$, strictly increasing subsequences $(m_k)$
 and $(n_k)$ in $\NN$ 
 exist such that
\begin{eqnarray}
     \label{eq:c+}
  c_{m_k} / c_{n_k} \rightarrow c,
\end{eqnarray}
then
   \begin{eqnarray}\label{eq:CS+}
 - \signXi  \langle \xi, \eta \rangle_\assoc 
    \le 2 \sqrt{\lnorm \xi \rnorm  \lnorm \eta \rnorm} / S_+
   \end{eqnarray}
If, additionally, $(\GF,\plus)$ is a group, let for $m\in\NN$
 \begin{eqnarray*}
    c_{-m}  &:=& \sup_{\xi \in \GF, \lnorm \xi \rnorm > 0}
             \frac{ \lnorm m (- \xi) \rnorm}{ \lnorm \xi \rnorm} < \infty,
    \\
   S_- & :=& \inf_{\footnotesize
               \begin{array}{c}
                 \xi, \eta \in \GF\\
                \signXi \langle \xi,  \eta\rangle_\assoc > 0
                  \end{array}
                 }
                 \liminf_{
                 m,n\rightarrow\infty}
                \frac{-\langle m \xi, n (- \eta)\rangle_\assoc}{\sqrt{c_m c_{-n}}
    \langle \xi,  \eta\rangle_\assoc} >0
    .
   \end{eqnarray*}
 If, for any $c>0$, strictly increasing subsequences $(m_k)$
 and $(n_k)$ in $\NN$ 
 exist such that
  \begin{eqnarray}
     \label{eq:c-}
  c_{m_k} / c_{-n_k} \rightarrow c
    ,
   \end{eqnarray}
then
\begin{eqnarray}\label{eq:CS-}
    \signXi \langle \xi, \eta \rangle_\assoc
  \le 2\sqrt{\lnorm \xi \rnorm  \lnorm \eta \rnorm} / S_-
  .
   \end{eqnarray}
   
\end{proposition}

\begin{proof}
  Note that the cases $-\signXi \langle \xi, \eta \rangle_\assoc \le 0$
  and $\signXi \langle \xi, \eta \rangle_\assoc \le 0$ in \eqref{eq:CS+} and \eqref{eq:CS-},
  respectively, are trivial and can be excluded henceforth.
  This is also true for $\lnorm \eta \rnorm =0$ by \eqref{eq:R0=0}.
  Otherwise, let $n \eta := |n| (-\eta)$ for $n < 0$.
  For $\varepsilon >0$ and $m,n \in \NN$ large enough, we have  
  \begin{eqnarray*}
    0 \le c_m^{-1} \lnorm m \xi \plus (\pm n) \eta \rnorm
    &\le & \lnorm \xi \rnorm +
      {\displaystyle\frac{c_{\pm n}}{c_m}}  \lnorm \eta \rnorm
      \mp S_\pm \sqrt{\displaystyle\frac{c_{\pm n}}{c_m}} (-\signXi \langle \xi , \eta
           \rangle_\assoc)
           ,
  \end{eqnarray*}
  where the top sign of `$\pm$'
  is taken if $\signXi \langle \xi , \eta
  \rangle_\assoc < 0$, and
  the bottom sign if $\signXi\langle \xi , \eta
  \rangle_\assoc > 0$ and $(\GF, \plus)$ is a group.
  Consider sequences of $n$ and $m$ such that
  $c_{\pm n} / c_m \rightarrow c := b^2\langle \xi, \eta \rangle_\assoc^2 / \lnorm
  \eta \rnorm^2$ for some $b > 0$. Then,
 \begin{eqnarray*}
    0 
    &\le & \lnorm \xi \rnorm  \lnorm \eta \rnorm -
           b(S_\pm - b) \langle \xi, \eta \rangle_\assoc^2
           ,
  \end{eqnarray*}
  so that for  any $b \in (0, S_\pm)$
   \begin{eqnarray*}
     \langle \xi, \eta \rangle_\assoc^2 
    &\le&  \frac{\lnorm \xi \rnorm  \lnorm \eta \rnorm}
          {b(S_\pm-b)}
          .
   \end{eqnarray*}
   Hence,
   \begin{eqnarray*}
    \mp \signXi \langle \xi, \eta \rangle_\assoc
    &\le&  2 \sqrt{\lnorm \xi \rnorm  \lnorm \eta \rnorm} / S_\pm
          .
   \end{eqnarray*}
\endproof\end{proof}

\begin{remark}
  \begin{enumerate}
  \item 
   If constants $c_+> 0$ and $a\in\RR\setminus \{0\}$ exist such that
   $$
   c_{ m} / m^{a} \rightarrow c_+
  \quad (m \rightarrow \infty)
  ,
  $$
  then Condition \eqref{eq:c+} is satisfied.
  If, furthermore $(\GF, \plus)$ is a group and $c_->0$ exists such
  that 
  $$
  c_{- m} / m^{a} \rightarrow c_-
  \quad (m \rightarrow \infty)
  ,
  $$
  with the same $a$ as above, then
  Condition \eqref{eq:c-} is satisfied.
  
\item
  If all conditions of Proposition \ref{prop:CS} are satisfied, then
  the Pearson correlation coefficient $\rho_{\xi, \eta}$, see
  Equation \eqref{eq:pearson},
  satisfies
  $$
  -\frac{|\aopt|}{S_+}\le
  \rho_{\xi, \eta}
   \le 
   \frac{|\aopt|}{S_-}
   .
   $$
   From the definition of $S_\pm$, it follows immediately that $|\aopt|/S_\pm
   =1$ in the case of a pre-Hilbert space.
   In general, however, $S_-$ and $S_+$ do not coincide and $|\aopt|/S_\pm$
   need not  be $1$. Hence, the simpler correlation coefficient
   $r_{\xi,\eta}$ in Remark \ref{rem:r}
   seems to fit a general framework better.
  \end{enumerate}
\end{remark}

\subsection{Constructions}
\label{sec:constructions}

Proposition \ref{prop:construction1} below shows that
if Equations \eqref{eq:sp.linear} and \eqref{eq:sp.symm} and 
and some
further assumptions hold, then an entropy exists with the
given kernel $\langle \cdot, \cdot \rangle$ as hemi-scalar product.

Proposition \ref{prop:scoring.rule} shows roughly that
for a given scoring rule  an entropy exists with the scoring rule
as associated hemi-metric. Reversely, under more restrictive assumptions
on the entropy, some of the associated hemi-metrics are scoring rules.

\begin{lemma}\label{lemma:sums.scalar.products}
  Let $(\GF , \plus)$ be a semigroup. Assume that $\langle \cdot , \cdot
  \rangle : \GF 
  \times \GF  \rightarrow \RR$ is a kernel with properties 
  \eqref{eq:sp.linear} and \eqref{eq:sp.symm}. Then, for all $\xi\in \GF $ and
  $i,j,m \in \NN$ we have
  \begin{eqnarray}\label{eq:scalarsum.1}
\langle i \xi, j \xi \rangle 
 &=&
 \sum_{k=1}^{i+j-1}  \langle  k \xi , \xi\rangle
 -     \sum_{k=1}^{j-1}  \langle k  \xi ,   \xi\rangle
-
     \sum_{k=1}^{i-1}  \langle k \xi,  \xi \rangle
     ,
 \\\label{eq:scalarsum.2}
    \sum_{n=1}^{m-1}  \langle n i \xi,  i \xi \rangle
    &= &
         \sum_{k= 1 }^{mi-1}  \langle k  \xi,  \xi \rangle
         -
      m \sum_{k=1}^{i-1}  \langle k \xi,  \xi \rangle.
  \end{eqnarray}

\end{lemma}

\begin{proof}
  Equation \eqref{eq:sp.linear} implies
  that 
  \begin{eqnarray*}
     \langle  (i-k) \xi , (j+k) \xi \rangle  =
    \langle (i-k-1) \xi , (j+k+1)\xi \rangle - \langle (i-k-1) \xi, \xi
    \rangle+
      \langle  \xi, (j+k)\xi\rangle
  \end{eqnarray*}
   for $0\le k  <i$. Summing up both sides from
   $k=0$ to $k=i-1$ yields \eqref{eq:scalarsum.1}.
   It follows immediately from \eqref{eq:scalarsum.1} that
    \begin{eqnarray*}\label{eq:i.ni}
      \langle  i \xi, ni \xi \rangle =
      \sum_{k= i  n}^{(n+1)i-1}  \langle k \xi,  \xi \rangle
      -
      \sum_{k=1}^{i-1}  \langle k \xi,  \xi \rangle, \qquad \forall \xi \in
      \GF ; i,n\in\NN
      .  \end{eqnarray*}
    Summing up yields \eqref{eq:scalarsum.2}.
\endproof\end{proof}

\begin{lemma}\label{lemma:sums.scalar.products.comm}
  Let the conditions of Lemma \ref{lemma:sums.scalar.products}
  hold 
  and assume that
 \begin{eqnarray}\label{eq:sp.commut1}
  \langle  \eta \plus \lambda , \nu \rangle & =&
  \langle  \lambda \plus \eta, \nu  \rangle
  \qquad \forall \lambda,\nu,\eta\in \GF 
   \\\label{eq:sp.commut2}
 \langle\xi,  \eta \plus \lambda \plus \nu \rangle & =&
  \langle \xi,  \lambda \plus \eta \plus\nu  \rangle
  \qquad \forall\xi, \lambda,\nu,\eta\in \GF 
  .
 \end{eqnarray}
 
  Then, for all $\xi, \lambda,\nu,\eta\in \GF$, we have
 \begin{eqnarray}
    \nonumber
   \langle \xi \plus \eta , \lambda\plus \nu \rangle
 &    = &
    \langle \xi \plus \lambda , \eta \plus \nu \rangle
 -   \langle   \xi , \eta   \rangle
         +   \langle \xi , \lambda \rangle
 +    \langle \eta , \nu \rangle
    -  \langle \lambda, \nu \rangle
          +  {}
    \\&&\label{eq:scalarsum.3}
     {} +  \langle \eta, \lambda \rangle - \langle  \lambda, \eta
       \rangle.
  \end{eqnarray}

  If, furthermore, $\langle \cdot, \cdot \rangle$ is symmetric, then
  for all $\xi, \nu\in \GF$ and $m\in\NN$,
  \begin{eqnarray}\label{eq:sumxieta}
   \sum_{k=1}^{m-1}\langle k (\xi \plus \eta), \xi \plus\eta \rangle
  =
  \langle  m \xi ,m \eta \rangle
  -m\langle   \xi ,\eta \rangle
 +\sum_{k=1}^{m-1}\langle  k \eta ,\eta \rangle
    +\sum_{k=1}^{m-1}\langle  k \xi ,\xi \rangle
    .
  \end{eqnarray}  
\end{lemma}
\begin{proof}
  To see \eqref{eq:scalarsum.3}, check the equalities
 \begin{eqnarray*}
   \langle \xi \plus \eta , \lambda\plus \nu \rangle
   -  \langle  \eta , \lambda\plus \nu \rangle
  & =&
   \langle \xi  , \eta \plus\lambda\plus \nu \rangle
   -
   \langle \xi  , \eta \rangle
   ,\\
 - \langle \xi \plus \lambda , \eta \plus \nu \rangle
   +  \langle  \lambda , \eta \plus \nu \rangle
  & =&
  - \langle \xi  ,\lambda \plus \eta \plus \nu \rangle
   +
   \langle \xi  , \lambda \rangle
   ,\\
   \langle  \lambda \plus \eta , \nu  \rangle -  \langle  \lambda ,
   \eta \plus \nu  \rangle
   &=&
 \langle  \eta, \nu \rangle   - \langle  \lambda, \eta
       \rangle
   ,\\
  - \langle \eta   \plus\lambda , \nu  \rangle +  \langle  \eta ,
   \lambda \plus \nu  \rangle
   &=&
 -\langle  \lambda, \nu \rangle   + \langle \eta, \lambda
       \rangle
       ,
  \end{eqnarray*}
 and sum them up.
 Equation \eqref{eq:scalarsum.3} implies
  $$
  \langle k (\xi \plus \eta), \xi + \eta \rangle
  =
  \langle ( k + 1)\xi ,(k+1) \eta \rangle
 -\langle  k \xi ,k\eta \rangle
 -\langle   \xi ,\eta \rangle
 +\langle  k \eta ,\eta \rangle
 +\langle  k \xi ,\xi \rangle
 .
$$
Summing up yields the second assertion of the lemma.
\endproof\end{proof}

Except that  the operator $\circ$ is not defined explicitly,
Proposition \ref{prop:construction1}
will give conditions, such that a kernel
$\langle\cdot,  \cdot \rangle$
can be interpreted as an associated hemi-scalar product on a structurally
simple space $\GF $. Proposition \ref{lemma:reverse.scalar} 
considers more general spaces $\GF $.
Crucial for the construction of $\lnorm \cdot \rnorm$
is the following consistency condition
\begin{eqnarray}
  \label{eq:Mxi}
   M_\xi :=
  \sup_{\eta\in \GF, m,n \in \NN, m \eta = n\xi}
\frac  e n \left(
    \sum_{k=1}^{m-1} \langle k \eta , \eta\rangle
    - \sum_{k=1}^{n-1} \langle k \xi , \xi\rangle
  \right) \in(0, \infty)
\end{eqnarray}
for some $e \in\{-1,+1\}$ and some $\xi \in \GF $.
Since $M_\xi$ is always non-negative, $M_\xi$ might be considered as
an entropy if $\eta \mapsto M_\eta$ is not identially $0$.
The consistency condition
states that $M_\xi$ may neither be $0$ nor
$\infty$.
Note that
Condition \eqref{eq:Mxi} does not
hold for $\xi \in \detsys$ and generally not  for both
signs $e$.

\begin{example}[Example \ref{bsp:hilbert} cont'd]\label{bsp:hilbert.2}
  For a pre-Hilbert space $H$ with
  scalar product $\langle \cdot , \cdot  \rangle_H$
  and $\eta = n\xi / m$,
  we have
   \begin{eqnarray*}
  \sum_{k=1}^{m-1} \langle k \eta , \eta\rangle_H
     - \sum_{k=1}^{n-1} \langle k \xi , \xi \rangle_H
     &=&
  \left(\frac{m (m-1)}2 \cdot \frac {n^2}{m^2} -
    \frac{n(n-1)}2
  \right) \langle  \xi , \xi \rangle_H
 \\& =& \frac{n}2\left(1-\frac nm   \right) \langle  \xi , \xi \rangle_H
  .   \end{eqnarray*}
  Hence, $e$ is necessarily $1$ and 
  we get $M_\xi = \langle \xi , \xi  \rangle_H /
  2$.
  That is, $\xi\mapsto M_\xi$ is indeed the standard entropy up to a
  factor $2$. Obviously, Equation \eqref{eq:Mxi} is far away from the standard
  interpretation of what a squared norm is.
\end{example}

\begin{example}[Example \ref{ex:realaxis}.\ref{ex:max.real} cont'd]
  For 
  $$
 \langle \xi,\eta \rangle =  ( \xi \wedge \eta)^\alpha
 ,\qquad \xi,\eta \in  [0,\infty)
 ,$$
 and $\eta = n\xi / m$, we have 
 $$
 \sum_{k=1}^{m-1} \langle k \eta , \eta\rangle
 - \sum_{k=1}^{n-1} \langle k \xi , \xi \rangle =
 \left((m-1) \frac {n^\alpha}{m^\alpha} - (n-1)
 \right) \xi^\alpha
 .$$
 If $\alpha = 1$, then both signs are possible in \eqref{eq:Mxi}.
 Otherwise, the sign $e$ is necessarily $-1$.
 In all  cases, we have $M_\xi =  \xi^\alpha$.
\end{example}

\begin{proposition}\label{prop:construction1}
  Let all the conditions within Lemma
  \ref{lemma:sums.scalar.products.comm}
  hold, i.e.,
  $(\GF , \plus)$ is a semigroup and $\langle \cdot , \cdot
  \rangle : \GF 
  \times \GF  \rightarrow \RR$ is a symmetric kernel with properties 
  \eqref{eq:sp.linear}, \eqref{eq:sp.symm}, \eqref{eq:sp.commut1}
  and  \eqref{eq:sp.commut2}.
  Assume further, that, for each
  $\nu,\eta\in \GF $, natural numbers $m,n\in\NN$ exist, such
  that $n\nu = m \eta$.
  Assume that a sign $e\in\{-1, 1\}$ and an element
  $\xi\in \GF$ exist such 
   that \eqref{eq:Mxi} holds. 
  Then, a non-negative function $\lnorm \cdot \rnorm$ exists on $\GF $,
   such that
   $e\langle \nu, \eta\rangle = \lnorm \nu \plus \eta \rnorm -\lnorm
   \nu\rnorm -\lnorm \eta \rnorm  $
   for all $\nu,\eta\in \GF$.
\end{proposition}
\begin{proof}
 The proof here and that of Proposition \ref{lemma:reverse.scalar}
  is structurally close to the proof of Lemma 27 in \cite{denker}.
    We show that the two defining conditions
    \begin{eqnarray}
      \label{eq:xiGmxi}
    \lnorm \xi \rnorm  &\ge &M_\xi,
    \\\nonumber
  \lnorm i \eta \rnorm & =&
             i \lnorm \eta \rnorm
             + e \sum_{k=1}^{i-1} \langle k \eta ,
                                \eta \rangle\quad \text{for }i\ge 1,
                                \eta \in \GF    
  \end{eqnarray}
  are sound and complete.
  For a fixed $\eta$ we have $n\xi = m \eta$ for some $m,n \in\NN$.
  Hence, necessarily,
  \begin{eqnarray}
    \label{eq:defeta}
     m \lnorm \eta \rnorm
   = n  \lnorm \xi \rnorm  + e \sum_{k=1}^{n-1} \langle k \xi,
   \xi \rangle - e \sum_{k=1}^{m-1} \langle k \eta,
   \eta \rangle,
  \end{eqnarray}
  and $\lnorm \eta \rnorm$ is non-negative for all $\eta \in \GF $.
  Assume that $\nu \in \GF$ and $n'\xi = m' \nu$. Then,
   \begin{eqnarray*}
      m m' \lnorm \nu \plus \eta \rnorm
   = (mn' + nm') \lnorm \xi \rnorm  + e \sum_{k=1}^{mn'+nm'-1} \langle k \xi,
   \xi \rangle - e \sum_{k=1}^{mm'-1} \langle k (\nu\plus \eta),
   \nu \plus \eta \rangle,
   \end{eqnarray*}
   so that, by \eqref{eq:defeta} and subsequently by
   \eqref{eq:scalarsum.1}  and \eqref{eq:sumxieta},
   we have
   \begin{eqnarray*}
  \lefteqn{  e  m m' ( \lnorm \nu \plus \eta \rnorm-   \lnorm \nu
     \rnorm-   \lnorm \eta \rnorm)
     }\\
      &=&   \sum_{k=1}^{mn'+nm'-1} \langle k \xi,
          \xi \rangle
          -\sum_{k=1}^{mn'-1} \langle k \xi,
        \xi \rangle
         -\sum_{k=1}^{nm'-1} \langle k \xi,
          \xi \rangle
          + {} \\  && {}
    + \sum_{k=1}^{mm'-1} \langle k \nu,
        \nu \rangle
        + \sum_{k=1}^{mm'-1} \langle k \eta,
        \eta \rangle
          -
          \sum_{k=1}^{mm'-1} \langle k (\nu\plus \eta),
          \nu \plus \eta \rangle
     ,
     \\
     &=&
         \langle m n'\xi, nm'  \xi \rangle
         -\langle  m m' \nu  ,m m'\eta \rangle
  + mm'\langle   \nu ,\eta \rangle
    =
  mm'\langle   \nu ,\eta \rangle.
   \end{eqnarray*}
   To see that the definition and the calculations are sound, we
   finally show that
   $$\lnorm ( mi ) \eta \rnorm = \lnorm m (i \eta) \rnorm
   .
   $$
   The left handside equals
   \begin{eqnarray*}
   \lnorm ( mi ) \eta \rnorm &=& (mi) \lnorm \eta \rnorm + e \sum_{k=1}^{mi-1} \langle k \eta ,
   \eta \rangle.
    \end{eqnarray*}
   The right handside equals
    \begin{eqnarray*}
   \lnorm  m (i \eta) \rnorm &=& m \lnorm i \eta \rnorm + e \sum_{k=1}^{m-1} \langle k i\eta ,
                                   i\eta \rangle
  \\&=&m \left(  i \lnorm \eta \rnorm
             + e \sum_{k=1}^{i-1} \langle k \eta ,
                                \eta \rangle \right)
           + e \sum_{k=1}^{m-1} \langle k i\eta ,
   i\eta \rangle=  \lnorm ( mi ) \eta \rnorm 
    \end{eqnarray*}
    by Lemma \ref{lemma:sums.scalar.products}.     
  \endproof\end{proof}
  
From a practical point of view, Proposition \ref{prop:construction1}
 deals essentially only with one-dimensional spaces that
have countably many elements. While the extension to more general
spaces by means of a continuity argument looks somehow
trivial, the consistency conditions for an extension to
higher dimensional spaces are unclear und Proposition
\ref{lemma:reverse.scalar} gives only an elementary statement.

 \begin{definition}    \label{def:equiv}
     Let $(\GF,\circ,  \plus, \lnorm .\rnorm)$ be a
     entropy-driven hemi-group with hemi-associative, comparable
     elements.
   The equivalence relation $\nu {\cal R}  \eta$ shall hold
  if and only if $n,m\in\NN$ exist such that $n \nu  = m \eta$.
  A set $\basis$ of equivalence classes of $\GF$ is called generating,
  if $\GF$ is 
  generated by $\basis$. 
  
  \end{definition}
  
  \begin{proposition}\label{lemma:reverse.scalar}
    Let the conditions of Definition \ref{def:equiv} hold.
      Assume that
      $$Z=\{ \varepsilon \in \GF: \langle\varepsilon , \varepsilon\rangle =0\}
     $$ 
     is a sub-semigroup of $(\GF , \plus)$
     and that
  \begin{eqnarray}
           \nonumber
    \langle \xi , \varepsilon \rangle
      &=&0\quad  \forall \xi, \in \GF, \varepsilon\in Z
          .            
      \end{eqnarray}
   For each $B \in \basis\setminus\{ Z\}$, let the conditions of Proposition
    \ref{prop:construction1} hold with some $\xi_B \in B$,
    for one and the same sign $e$. 
    For $\eta \in B$ let $m_B(\eta), n_B(\eta) \in \NN$ such that 
    $m_B(\eta) \eta = n_B(\eta) \xi_B$.
    Let $\lnorm \xi_B  \rnorm \ge M_{\xi_B}$, $B\in\basis$, and 
    $$
    \lnorm a_B \rnorm =   \frac{1}{m(a_B)} \left(n(a_B)  \lnorm \xi_B \rnorm  + e \sum_{k=1}^{n(a_B)-1} \langle k \xi_B,
     \xi_B \rangle - e \sum_{k=1}^{m(a_B)-1} \langle k a_B,
     a_B \rangle\right)
,\quad
$$
 for all $a_B \in  B$ and all $ B\in\basis$.
  Assume further,
  that for all  $a_B, b_B \in B $, $B\in \basis$, $\ell \in \NN$,
  the inequality
  \begin{eqnarray}\label{eq:a>0}
  \sum_{i=1}^\ell  \lnorm a_{B_i} \rnorm +
    e\sum_{i=1}^{\ell-1} \langle a_{B_i},  \sum_{j=i+1}^\ell
  a_{B_j}\rangle \ge0
   \end{eqnarray}
   shall hold and the equality
  $$
  a_{B_1} \plus  \ldots\plus  a_{B_\ell} = b_{B'_1} \plus  \ldots\plus  b_{B'_m}
  $$
  shall imply that
  \begin{eqnarray}\label{eq:a=b}
   \sum_{i=1}^\ell  \lnorm a_{B_i} \rnorm +
    e\sum_{i=1}^{\ell-1} \langle a_{B_i},  \sum_{j=i+1}^\ell
    a_{B_j}\rangle
  =
 \sum_{i=1}^m  \lnorm b_{B'_i} \rnorm +
e  \sum_{i=1}^{m-1} \langle b_{B'_i},  \sum_{j=i+1}^m
    b_{B'_j}\rangle
    .
  \end{eqnarray}
  Then,
  a nonnegative function $\lnorm .\rnorm$ exists, such that
  \begin{eqnarray}\label{eq:prim.sp}
   e \langle \nu , \eta\rangle =
         \lnorm
    \nu  \plus \eta \rnorm - \lnorm \nu \rnorm - \lnorm \eta \rnorm 
    \quad \forall \nu,\eta\in\GF.
  \end{eqnarray}
 
\end{proposition}
\begin{proof}
  If $Z$ is an equivalence class, define
  $\lnorm \varepsilon \rnorm =0$ for all $\varepsilon \in Z$, so that 
  \eqref{eq:prim.sp}
  holds on $Z$. Otherwise start with the equivalence class that
  contains $Z$.
  Let $U$ be a sub-semigroup of $\GF$, generated by the equivalence
  class
  that contains $Z$ and some other
  elements of $\basis$.
  Assume that $U\subsetneq \GF $ is a largest sub-semigroup
  such that \eqref{eq:prim.sp} holds, i.e., 
  which cannot be extended to a larger sub-semigroup.
  We will get a contradiction by constructing an extension.
  Let $A \in \basis$ such that $A \cap \GF \setminus U $ is
  not empty.
  For $\eta \in U$ and $a_A\in A$ we define
  $$
  \lnorm a_A \plus \eta \rnorm := e \langle a_A, \eta\rangle +
   \lnorm a_A \rnorm + \lnorm  \eta \rnorm 
   .$$
   The definition is sound by \eqref{eq:a=b} and $
  \lnorm a_A \plus \eta \rnorm\ge0$ by \eqref{eq:a>0}.
\endproof\end{proof}

\begin{definition}\label{def:scoring.rule}
  Let $\GF$ be a non-empty set. A non-negative, non-constant kernel  $S : \GF \times \GF
  \rightarrow [0,\infty)$ is called a \emph{proper scoring rule},
  if $S(\xi,\xi)=0$ for all $\xi \in \GF$.
\end{definition}

 \begin{remark}
   The two standard definitions of a
   proper scoring rule  $S$ are that $S(\eta, \xi)
   \ge S(\xi, \xi)$ for all $\eta,\xi \in \GF$,
   or, alternatively, that $S(\eta, \xi)
   \le S(\xi, \xi)$  for all $\eta,\xi \in \GF$, where
   $G$ is a set of distributions. The map
   $(\eta,\xi) \mapsto S(\eta, \xi) - S(\xi, \xi)$
   is called the  \emph{associated divergence function} 
   of the scoring rule
   \citep{gneitingraftery07}.
   Obviously,  the associated divergence function is 
   a  proper scoring rule with $S(\xi,\xi)=0$ for all $\xi \in \GF$.
 \end{remark}

\begin{proposition}\label{prop:scoring.rule}
  Let $\GF$ a non-empty set, $\plus$ a binary operator
  and $\lnorm .\rnorm$ a non-negative function on $\GF$.
  Let
  \begin{eqnarray}
    \label{eq:scoring.rule}
    A =  \{ a \in \RR :  a(\lnorm \eta\plus \xi \rnorm 
    -\lnorm \xi \plus \xi\rnorm + \lnorm \xi \rnorm - \lnorm  \eta\rnorm) \ge
    \lnorm  \xi\rnorm - \lnorm \eta \rnorm \quad \forall \xi,\eta \in \GF\}
  \end{eqnarray}
  and let $r_a(\eta, \xi) := a\lnorm\eta \plus \xi \rnorm
  + (1-a) (\lnorm \xi \rnorm + \lnorm \eta\rnorm)$.
  Then, for all $a\in A$, the map
  $$
  (\eta, \xi) \mapsto r_a(\eta, \xi) - r_a(\xi, \xi)
  $$
  is a proper scoring rule.
  Reversely, let $S$ be a proper scoring rule  on $\GF \times \GF $
  for a non-empty set $\GF$
  such that 
  \begin{eqnarray}
    \label{eq:supS}
    \sup_{\xi,\eta, S(\eta,\xi) > 0}
    \frac{S(\eta, \xi) }{ S(\omega, \eta)
    +  S(\omega, \xi)}
    < \infty      
  \end{eqnarray}
  holds for some $\omega \in \GF$.
  Then, a semigroup
  $(\GF',\plus)$ 
  and a non-negative function $\lnorm .\rnorm : \GF'
  \rightarrow [0,\infty)$
  and an embedding $\iota : \GF \rightarrow \GF'$
  exist,
  such that $A\cap \Xi\not= \emptyset$ and
  $S(\eta,\xi) =\rho_{a}(\iota(\eta), \iota(\xi))$ for some 
  $a \in A\cap \Xi$.
\end{proposition}

\begin{proof}
  Part 1 is obvious.
  Reversely, let \eqref{eq:supS} hold.
   Let $ \GF' = \GF \times \GF$ and identify $\GF$ with $\{\omega\} \times \GF
   \subset \GF '$. Let for $(\nu, \eta), (\lambda, \xi)\in \GF '$
    $$
   (\nu,\eta)\plus (\lambda, \xi) := \left\{
     \begin{array}{ll}
       ( \lambda,\xi), & \nu = \eta= \omega\\
       ( \nu,\xi), &\nu \not =\omega\\
      (\eta,\xi),   &  \nu  =\omega, \eta \not = \omega 
       \end{array}
    \right.
    .
    $$
    so that $(\omega, \omega)$ is a left neutral
    element of the semigroup $(\GF ',\plus)$.
    Let
    \begin{eqnarray}\label{eq:sf2}
    \lnorm (\eta, \xi) \rnorm
    &:=& a^{-1} S(\eta, \xi)
         + (1 - a^{-1})(S (  \omega, \eta)+ S ( \omega, \xi))
    \end{eqnarray}
    for some $a \in\RR\setminus\{0\}$.
    If
    \begin{eqnarray}\label{eq:sf3}
            a &<&  1 -  \sup_{\eta,\xi, S(\eta, \xi) > 0}
    \frac{S(\eta, \xi) }{ S( \omega, \eta)
                  +  S( \omega, \xi)}  
                  .
    \end{eqnarray}
    then $a$ is negative, 
    the map
    $\lnorm . \rnorm$ is non-negative, and
    $\rho_a((\omega, \eta),( \omega, \xi)) = S(\eta, \xi)$.

    \note{\color{red}
    We have
    \begin{eqnarray*}
      a \lnorm (\omega, \eta) \rnorm
      &= &
           S_*(\omega, \eta)-
           (1-a)(S_*(\omega, \omega) + S_*(\omega, \eta))
      \\&=& a  S_*(\omega, \eta) + (1-a) S_*(\omega, \omega)
     \end{eqnarray*}  
    \begin{eqnarray*}
      \rho_a((\eta, \omega),(\xi, \omega))
      &= &
      a  \lnorm (\eta, \xi) \rnorm  + (1- a) (    \lnorm (\omega, \eta)
           \rnorm   +            \lnorm (\omega, \xi) \rnorm)
      \\&=& S_*(\eta, \xi) - (1-a) (S_*(\omega, \eta)
            +  S_*( \omega, \xi))
      \\&&+ \frac{1-a}a (a  S_*(\omega, \eta) - (1-a) S_*(\omega,
           \omega) + a  S_*(\omega, \xi) - (1-a) S_*(\omega, \omega))
      \\&=&   S_* (\eta, \xi) - 2\frac{(1-a)^2}aS_*(\omega, \omega)
    \end{eqnarray*}
    Further,
   \begin{eqnarray*}
      a \lnorm (\omega, \eta) \rnorm(S (  \omega, \eta)+ S ( \omega,
     \xi))
     &=& \frac{S(\eta, \xi)}{S (  \omega, \eta)+ S ( \omega, \xi)}
         - 1 + a \le 0
   \end{eqnarray*}
   iff the given condition is satisfied.

   FUNKTIONIERT NICHT
   
     More generally, let $S_*(\eta, \xi) = S(\eta, \xi) + f(\xi)$
   for some function $f : \GF \rightarrow \RR$ and $S$ a proper scoring
   rule
   in the sense of Definition \ref{def:scoring.rule}.
   Replacing $S$ by $S_*$
   in \eqref{eq:sf2} and \eqref{eq:sf3},
   it can be easily seen that $A$ is independent of $f$,
   and  $\lnorm . \rnorm$ remains non-negative as long as
   $f$ is non-negative,
   \eqref{eq:supS} is satisfied and $a<0$.
   E.g. for the CRPS,  $S_*(\xi, \xi)$
     
       \begin{eqnarray}
    \lnorm (\eta, \xi) \rnorm_*
    &:=& a^{-1} (S(\eta, \xi) + f(\xi) 
         - (1-a)(S (\omega, \xi) + f(\xi) + S (\omega, \xi) + f(\xi)))
         \\=&=&\lnorm (\xi, \xi) \rnorm + f(\xi) - \frac{1-a}a f(\xi)
      .
    \end{eqnarray}
   
  }
\endproof\end{proof}

   Note that
   $r_a$ in Proposition \ref{prop:scoring.rule}
   is real valued, in general. To obtain non-negativity, an
   additional condition
   is necessary.
   For instance, if 
   $\lnorm \eta \plus  \xi\rnorm  \le  (\lnorm \eta
  \rnorm  +\lnorm   \xi\rnorm ) ({a-1})/a$
  and $2 \lnorm \xi \rnorm (a-1) / a \ge  \lnorm \xi \plus \xi \rnorm $
  for some $a <0$, then
  $a \in \Xi \cap A$.

  \note{
    Condition \eqref{eq:supS} should be satisfied for the brier score
    with $\omega = (1,0,0, \ldots)$.
  }

\section{Transformation of systems: scale invariance}
\label{sec:trafo}
  Up to now, we have introduced 
   the operator $\circ$, which allows to consider two independent systems, and
   the 
   operator $\plus$, which  allows to compare two different systems.
   In many practical applications, we need
   a third operation, denoted by $\cdot$, which models
   certain modifications of a given system.
   For instance, the  principal component
   analysis
   (PCA)
   projects data from a
   higher dimensional space into a lower dimensional one in some
   optimal way.
   In the following, $(\GF , \cdot)$ will model transformations of a
   system,
   that are necessary
   to deal with mappings such as  projections in a PCA \citep{schlatherreinbott21}.

\begin{definition}\label{def:entropy.hemi-ring}
   Let $(\GF,\circ,  \plus, \lnorm .\rnorm)$ be a
   entropy-driven hemi-group with 
   comparable
     elements and $\GF =\GFsup$ in Definition
     \ref{def:entropy.hemi-group}.
     Let $\cdot$ be an additional binary operator.
     An element $\xi\in \GF$ is called a
  \emph{deterministic transformation}
  if 
  $$\lnorm \xi  (\nu \circ \eta) \rnorm =
  \lnorm \xi  \nu
  \circ  \xi  \eta\rnorm
  \quad \forall \nu,\eta \in \GF
.  $$
  Let $\dettrafo $ be the set of all deterministic
  transformations and assume, 
  that
  $$
  \detsys
  \subset \dettrafo 
  .$$
   We assume further that 
   $\xi \nu , \nu \xi \in \detsys$ for all $\nu \in \dettrafo, \xi \in \detsys$.    
  Let the operators $\plus$ and  $\cdot$ be  right hemi-distributive
     i.e.,
    $$
    \lnorm (\xi \plus \nu) \eta \rnorm = 
    \lnorm \xi \eta \plus \nu \eta \rnorm
    \qquad \forall \xi,\eta, \nu \in \GF 
    .$$
  Then,
  $(\GF,\circ, \plus, \cdot, \lnorm .\rnorm)$ is called an
  entropy-driven hemi-ring.
  A deterministic transformation $e \in \dettrafo$ is called a  \emph{left entropy-invariant
    transformation} if 
  $$
  \lnorm e\eta \plus e\nu \rnorm =
  \lnorm \eta \plus \nu \rnorm
  \quad
  \forall \nu,\eta\in\GF
  .
  $$
 Let $\GF_e\subset \dettrafo$ be the set of all left entropy-invariant transformations.
 \end{definition}

Let $x_1,\ldots, x_n$ be a set of $d$-dimensional data and $\xi \in \RR
\setminus\{0\}$.
In  PCA, 
the rescaled data $\xi x_1,\ldots, \xi x_n$ are subject to the same
optimal projection map (onto a $k$-dimensional hyperplane with $k<d$)
as the original data $x_1,\ldots, x_n$.
The algebraic representation of this property is called
\emph{scale invariance} in the following definition.

 \begin{definition}\label{def:scale.invariant}
  Let $\GF$ be an entropy-driven hemi-ring.
  An element $\xi \in \GF$ is called a left rescaling if
  a constant  $c_\xi \ge 0$ exists such that
   $$
  \lnorm \xi \nu\rnorm =
  c_\xi
  \lnorm \nu \rnorm
  \quad
  \forall \nu\in\GF\  
  .$$
  If $
  \lnorm \nu\xi \rnorm =
  \tilde c_\xi
  \lnorm \nu \rnorm
  $ for all $\nu\in\GF$ and some $\tilde c_\xi \ge 0$ and all $\xi\in \GF$,
  then $\xi$ is called a \emph{right
    rescaling}.
  Let  $\GF_l$ ($\GF_r$) be the set of left (right) rescalings.
  If $\GF = \GF_l$ 
  and $c_\xi = c \lnorm \xi\rnorm$ for some $c>0$, then both  $\GF$ and its entropy are called
 scale-invariant.   

  \end{definition}

 \begin{remark}
   Invariance is not a property that appears incidentially, but
   is driven by practical considerations and the entropy has to be
   defined accordingly. The terms ``determinstic transformation'' and
   ``determininistic system'' are abstract, but should cover the
   practical point of view in many cases if we identify a constant map
   onto a deterministic system with the deterministic system itself.
 \end{remark}

 \subsection{Properties}

 \begin{lemma}\label{rem:rescaling}
   Let $\GF$ be an entropy-driven hemi-ring.
   Then, the following assertions hold.
   \begin{enumerate}
   \item For $e\in \GF_e$ and $\xi \in \GF$ we have $\lnorm e\xi\rnorm =\lnorm \xi\rnorm$.
   \item 
     Assume that the
     set of rescalings $\GF_{rl} =\GF_r \cap \GF_l$ is not empty.
     If $\lnorm \nu \rnorm >0$ for some $\nu \in \GF_{rl}$, then
     a constant $c\ge 0$ exists such that
     for all $\xi \in \GF_{rl}$ we have
     $c_\xi = c \lnorm \xi \rnorm$.
     If, additionally, a left entropy-invariant transformation $e$
     exists, then $c= \lnorm e  \rnorm ^{-1}>0$.
   \item 
     If $\GF $ is scale-invariant, then $\GF_r=\GF $.
     Reversely, if $\GF =\GF_r = \GF_l$ and
     a left entropy-invariant transformation
     exists, then $\GF $ is scale-invariant.
  \end{enumerate}

 \end{lemma}
  \begin{proof}
    \begin{enumerate}
    \item For all $\xi\in\GF$, $\varepsilon \in \detsys$ and $e \in \GF_e$ we have
      $\lnorm e \xi \rnorm = \lnorm e \xi \plus e \varepsilon \rnorm =
      \lnorm  \xi\rnorm$.
    \item
      For all $\xi,\nu\in\GF_{rl}$ 
      we have
  $ \lnorm \xi \nu \rnorm = c_\xi  \lnorm \nu \rnorm = \tilde
      c_\nu  \lnorm \xi \rnorm
      $.      
      If $\lnorm \nu \rnorm >0$ then
      $ c  = {\tilde c_\nu}/{\lnorm  \nu \rnorm} $
      is a suitable choice.
      If $e\in\GF_e$, then,
      $
      0 < \lnorm \nu \rnorm = \lnorm e \nu \rnorm
       =  \lnorm e\rnorm \cdot c \lnorm \nu \rnorm
       $
       and neither $c$ nor  $\lnorm e\rnorm$ can
       be $0$.
    
     \item For all $\nu \in \GF $ we have
      $$
      \lnorm \xi \nu \rnorm =
      \lnorm \xi \rnorm \cdot c  \lnorm \nu \rnorm
      \qquad
      \forall \xi \in \GF 
      ,
      $$
      so that $\GF_r = \GF $.
      The reverse statement follows immediately from the first part.
     \end{enumerate}
  \endproof\end{proof}

\begin{proposition}\label{prop:scalar2}
  Let $(\GF,\circ, \cdot, \plus, \lnorm .\rnorm)$ be an
  entropy-driven  hemi-ring with comparable elements.
  If, additionally, $\GF $ is scale-invariant
  and a left entropy-invariant transformation $e$  exists, then
   the function
  $\xi \mapsto \lnorm \xi \plus \xi \rnorm - 2\lnorm \xi \rnorm$
  does not change sign and
  \begin{eqnarray}
    \label{eq:rho.scaling}
    \lnorm e \rnorm \rho_a(\xi \nu, \eta \nu)
    & = &  \lnorm \nu \rnorm \rho_a(\xi, \eta ), \quad \forall a \in \Xi
    ,\\\label{eq:scalar.scaling}
   \lnorm e \rnorm \langle \xi \nu, \eta \nu \rangle_*
    & = &  \lnorm \nu \rnorm \langle \xi, \eta \rangle_*
          \qquad \mbox{for } \langle \cdot , \cdot  \rangle_*= \langle \cdot , \cdot  \rangle,\;
          \langle \cdot , \cdot  \rangle_2 \mbox{ or } \langle \cdot , \cdot  \rangle_\assoc
    ,\\\nonumber
    \langle \xi , \xi \rangle_\assoc =0 &\Leftrightarrow & \lnorm \xi \rnorm = 0
                                                \mbox{ or }  \langle
                                                \eta ,  \eta \rangle_\assoc
                                                = 0 \;\forall \eta \in
                                                           \GF
                                                           .
  \end{eqnarray}
   If, furthermore, $ \lnorm e \plus e \rnorm  \not= 2\lnorm e \rnorm$, then
  the following assertions hold:
  \begin{eqnarray}\nonumber\label{eq:rho1}
    \rho_a(\xi ,\xi)
    &=& 0 \quad \forall \xi \in \GF, \mbox{if } a^{-1} = 1- \frac{\lnorm e \plus e
        \rnorm }{2 \lnorm e \rnorm},
        \note{\hbox{$a\in \Xi$ always!}}
    \\ \label{eq:scalar.Gs}
\langle \xi, \xi\rangle_\assoc=0
                                              &\Leftrightarrow& \xi
                                                                \in \GF_0
  \end{eqnarray}
\end{proposition}
\begin{proof}
  We have
  $$
    \lnorm \xi \plus  \xi \rnorm= 
  \lnorm e\xi \plus  e\xi \rnorm= 
 \lnorm (e \plus  e)\xi \rnorm= 
 \lnorm e \rnorm^{-1}\lnorm e \plus  e\rnorm\, \lnorm \xi \rnorm
 .
 $$
 The remaining assertions are immediate.
\endproof\end{proof}

\begin{proposition}\label{pprop:scale}    
  If $\lnorm \cdot \rnorm$ is scale invariant, then the
  following properties hold:
  \begin{enumerate}
  \item If $\GF_0=\{0\}$, then the hemi-group $\GF$ is division free, i.e.,
    for all $\mu, \nu
    \in \GF$ with
    $\mu  \nu = 0$ we have $\mu = 0$ or
    $\nu = 0$.
  \item\label{pprop:scale.2} 
    Let $\GF \subset \RR$ be a non-trivial interval
    with standard topology and $1\in \GF$.
    If $(\GF,\cdot)$ is the standard multiplication
    with   multiplicative neutral element $1_\GF=1$
    and $\lnorm 1_\GF\rnorm$ is given,
    then $\lnorm \mu\rnorm = b |\mu|^\alpha$ 
    for some unique $\alpha\in\RR\setminus\{0\}$
    and some unique $b > 0$.
   \end{enumerate}
 \end{proposition}
 \begin{proof}

  \begin{enumerate}
  
\item
      $0 = \lnorm \mu \nu  \rnorm = c \lnorm \mu  \rnorm \lnorm \nu
      \rnorm$ implies that 
      $\lnorm \mu  \rnorm =0$ or $\lnorm \nu
   \rnorm=0$.  
\item
   Let $b=\lnorm 1_\GF  \rnorm$.
   Without loss of generality $b=1$ as $b$ cannot be $0$.
   Let $A = \{ \log(x) : x\in  \GF \cap (0,\infty)\}$.
   The function $\ell:A\rightarrow\RR$, $ x\mapsto \log \lnorm e^x \rnorm$ is
   well defined on some nontrivial interval that includes $0$
   and is continuous there.
   Since $\ell$ obeys Cauchy's functional equation we get
   $\lnorm \mu \rnorm = \mu^\alpha$ for $\mu \in \GF \cap (0,\infty)$
   and some $\alpha\in\RR$. Now, assume that
   $\GF\cap(-\infty,0)\not=\emptyset$.
   Then, Cauchy's functional equation delivers that
   $\lnorm \mu \rnorm = |\mu|^\beta$ for $\mu \in \GF \cap (-\infty,0)$
   and some $\beta\in\RR$.
   For $\mu \in \GF \cap [-1,0)$
   we have  $|\mu|, \mu^2 \in \GF$, so that
   $\lnorm \mu\rnorm^2 = \lnorm \mu^2 \rnorm  = \lnorm |\mu|^2 \rnorm
   = \lnorm |\mu| \rnorm^2$. 
   Hence, $\alpha=\beta$.
   The entropy definition yields $|\mu \circ \nu|^\alpha = |\mu|^\alpha +
   |\nu|^\alpha$ with $\alpha\in\RR$.
   Assume $\alpha =0$. Then, 
   $ |1_{\GF} \circ 1_{\GF}|^0 \not = 1 + 1$ in contradiction to
   the definition of an entropy.
   Now, assume that  $\lnorm \cdot \rnorm_\alpha: \mu \mapsto
   \mu^\alpha$ and  $\lnorm \cdot \rnorm_\beta:\mu 
   \mapsto \mu^\beta$ 
   are two scale invariant entropies with
   $\alpha,\beta\in\RR\setminus\{0\}$. Then, for all $\mu \in \GF\setminus\{0\}$,
   $$
   (1 + |\mu|^\alpha )^{1/\alpha} = |1_{\GF}  \circ \mu| = (1 + |\mu|^\beta )^{1/\beta}
  , $$
   so that $\alpha=\beta$.
   By a continuity argument, we get $\lnorm 0 \rnorm =0$, if $0\in \GF$.
  \end{enumerate}

\endproof\end{proof}

\section{Outlook: Entropy-driven statistical modelling}
\label{sec:stat.model}
 For consistency with many statistical approaches, we 
 restrict to Polish spaces from the beginning,
 although a topological space would be sufficient, at first.

\begin{definition}
  Let $(\Omega, \A)$ and $(\GX, \B)$  be
  Polish spaces.
  Let $(\bigtimes \Omega, \bigtimes \A, \bigtimes
  \PP)$
  be the product
  probability space  provided
  with the cylinder $\sigma$-algebra.
  Let denote by
  $\pi_k$ the $k$-th canonical projection and let
  $Y=(X_1(\pi_1),X_2(\pi_2),\ldots) : \bigtimes \Omega \rightarrow
  \bigtimes \GX $ be measurable,
  i.e., $X_i : \Omega \rightarrow \GX $ are measurable maps.
  Then, $ Y_k$ and $Y_i $
  are called \emph{strictly independent} if and only if $k\not=i$.
\end{definition}

    Let $\X$ be a set of measurable mappings from $(\Omega, \A)$ into
    $\GX$.
In the following, the mapping
$\xi : \X \rightarrow \X$ shall 
be applied to the $k$-th component of a random vector
$X : \bigtimes \Omega \rightarrow \bigtimes \GX $, i.e.,
$$
(\xi, k) : \bigtimes \X \rightarrow \bigtimes \X,\qquad X \mapsto \tilde X
$$
with
$$
\tilde X_i(\pi_i) = \left\{
  \begin{array}{ll}
    X_i(\pi_i), & i\not= k\\
    (\xi X_k)(\pi_k) ,
         & i = k
  \end{array}
  \right.
.  $$
We will
 cease the underlying product space for ease, 
and write shortly $\xi X_k$, for instance.

\begin{definition}\label{def:stat.model}
  Let  $(\GF, \circ, \plus, \cdot,\lnorm .\rnorm)$
  be an entropy-driven hemi-ring with left entropy-invariant
  transformation $1\in \GF$.
  Let $(\GX, \plusG)$  be a Polish semigroup, $(\Omega, \A, \PP)$
  a probability space,
  $X_0 : \Omega \rightarrow \GX$ a measurable function,
  and $\PP_0$ its push forward measure.
  The ensemble $(\Omega, \GX, \GF )$
is called a
  \emph{entropy-driven statistical model},
  if the following conditions hold:
  \begin{enumerate}
  \item    
    Let $\X$ be a set of measurable mappings from $(\Omega, \A)$ into
    $\GX$.
    Let there be an isomorphism between $\GF$ and 
    a set of 
    maps on $\X$, so that we can identify  an element of
    $\GF$ by its corresponding map.
    We assume, that $\X = \{ \xi X_0 : \xi \in \GF\}$,
    and $1X =X$ for all $X \in \X$.
  
 \item For all $\xi \in \GF $, the random variable
    $\xi X_0$ is deterministic, if and only if 
    $\xi$  is a deterministic system.
  \item  For all $\xi,\nu\in\GF$, the coincidence of distributions
    $\xi X_0\sim \nu X_0$ implies
    $\lnorm \xi \rnorm = \lnorm \nu \rnorm $.
 
  \item
      Let  $Y_0\sim X_0$ 
    be  strictly independent of $X_0$.
      Then, for all $\xi,\nu \in\GF$, we shall have
    $$
    \xi X_0 \plusG \nu Y_0
    \sim (\xi \circ \nu) X_0
    .
    $$
   \end{enumerate}
   We  call the ensemble $\cal P$ of probability distributions generated by $\xi X_0$,
  $\xi \in \GF $, a \emph{stable set of distributions}.
\end{definition}

 Note that the 
 entropy is here a  relative quantity with respect to
 the entropy of $X_0$
  and not an
  absolute one.
 An entropy-driven hemi-ring can have different
corresponding statistical models
as the following examples demonstrate.
This shows, as emphasized in \cite{mccullagh2002statistical}, that
the skeletal description is most important, and the fullfilling
family of distributions is secondary.

In the following examples, $\PP$ can be chosen to be $\PP_0$, so that
$X_0$ is the identity.

\begin{example}[Stable distributions]
  \label{ex:stable.stat}
  Let $[0,\infty)\subset \GF \subset\RR$ be provided with the standard topology
   and the standard operators $+$ and $\cdot$.
   Let the entropy be scale invariant, 
   i.e.,
     $\lnorm \xi \rnorm = |\xi|^\alpha$, $\xi \in \GF$,
   for some fixed $\alpha \in(0,2]$
   according to Proposition \ref{pprop:scale}.\ref{pprop:scale.2}.
   Then, $\GF_0 = \detsys = \{0\}$  is indeed the set of deterministic systems.
   Furthermore, $\GF_e = \{-1,1\} \cap \GF $
   and the binary operator $\circ$ obeys the equality
   $|\xi \circ \nu| = (|\xi|^\alpha + |\nu|^\alpha)^{1/\alpha}$.
  Let 
  $\xi : X \mapsto \xi X$ be the  multiplication
  of a random variable $X$ by a scalar $\xi\in\GF$, and 
  let $(\GX, \plusG) = (\RR, +)$.
  Let us consider the following two cases.
  See Example \ref{ex:realaxis} for the implications of the different
  choices for $\GF$.
  \begin{description}
  \item[$\GF= [0,\infty)$.]\   \\  
    An entropy driven statistical model is obtained by choosing
    $\PP_0$ as  a symmetric $\alpha$-stable distribution
    $S_\alpha(\sigma, 0, 0)$ with some fixed $\sigma > 0$.
    Note that the
     covariation norm  $\|X\|_\alpha:=\sigma$,
     see \cite{samorodnitskytaqqu},  is closely related to
     the entropy. Note also that the  Gaussian case is  included as
     $S_2(\sigma, 0, 0) = \mathcal{N}(0, 2 \sigma^2)$
     and the entropy is here the variance (up to factor).
     Clearly, another standard choice in the Gaussian case 
     is $\PP_0 = S_2(2^{-1/2}, 0, 0)$.
  \item[$\GF = \RR$.]\ \\
    For $\alpha \in(0,2]$, an entropy-driven statistical model
    is obtained by any symmetric $\alpha$-stable distribution
    $S_\alpha(|\sigma|, 0, 0)$ with $\sigma \in\RR\setminus\{0\}$.     
  \end{description} 
\end{example}

In the previous example, the different entropy driven statistical
models only differ in the chosen scale in $\PP_0$. In the next example,
also $(\GX,\plusG)$ changes.

  \begin{example}[Max-stable distributions] \label{ex:maxstable.stat}
    Let $\alpha > 0$ be fixed and
    $G=[0,\infty)$ be provided with the standard topology
    and the operators $\vee$ and $\cdot$.
    Most of the general comments in
    Example \ref{ex:stable.stat} apply, except that 
    $(\GX,\plusG)=([0,\infty), \vee)$ and $\PP_0 =
    \Phi_{\alpha,\lambda}$, where $\Phi_{\alpha,\lambda}$ equals
    the Fr\'echet distribution,
    $\Phi_{\alpha,\lambda}(x) = \exp( - \lambda^\alpha / x^\alpha) \1_{x >
      0}$
    with $\lambda > 0$.
     Another entropy-driven statistical model
     is given by the min-stable distributions,
     where $(\GX,\plusG)=((-\infty,0], \wedge)$ and $\PP_0 = \Psi_{\alpha,\lambda}$
     is a  Weibull distribution,
     $\Psi_{\alpha,\lambda}(x) = 1 \wedge \exp( - (- \lambda x)^\alpha)
      $.
    
\end{example}

\note{
All stochastic examples in the preceding sections
can be reformulated in terms of Definition \ref{def:stat.model}.
We sketch some of them.
In Example \ref{ex:basic}.1,
the set $\X$ is set of Gaussian random variable including the Dirac
measure
$\delta_0$;
$X_0$ can be any element of $\X$ except the Dirac measure;
$\lambda$ is a map $\lambda : \X \rightarrow \X, X \mapsto \lambda X$.
}

\section{Conclusions}
This article shows that a geometrical approach to the scalar product
through the Pytha\-gorean theorem
allows the unification of many notions in statistics and a deeper
reasoning for the choice of several risk functions.
We have introduced three different operators $\circ$, $\plus$ and
$\cdot$.
While $\circ$ is mainly needed for formal reasons only, the operator
$\plus$ is essential for any kind of modeling and is the most
difficult one to define in practical applications.
The counterparts in the statistical model also differ: $\circ$ is
related to equality in distribution, whereas $\plus$ is related to
equality almost surely. Similarly, $G_0$ is related to ``$0$ with
probability 1'',
whereas $\detsys$ refers to deterministic
objects.
 Scale invariance, defined through
the operator $\cdot$, is a rather strong property, which implies, in
simple cases, uniqueness of the entropy.
This result, however, leads already to a crucial progress in developing an
adequat PCA for $\alpha$-stable and max-stable distributions
\citep{schlatherreinbott21}.
We conclude by resuming some results from the generalized
hemi-group set-up, which might be unexpected:
\begin{itemize}
\item In the hemi-commutative law, three operators are
envolved, similar to the associative law; hence, a left commutative and a
right commutative law exist, cf.~Remark~\ref{rem:hemi}.\ref{rem:hemi.hemi}
and Equations \eqref{eq:sp.commut1} and \eqref{eq:sp.commut2}.
\item
  Both the Kullback-Leibler divergence and the mutual information are,
  at the same time, a scalar product and a hemi-metric.
\item
  Maximum likelihood has a clear Bayesian interpretation, but not the
  Tichonov regularization.
\item
  Bivariate Gaussian vectors as a whole can be, at the same time,
  fully dependent and uncorrelated.

\item
  Linearity of the scalar product is not needed to proof the
  Cauchy-Schwarz inequality.
  For a hemi-scalar product,
  the elementary property of linearity is reduced to the equality
  $$
  \langle  \xi \plus \eta , \nu  \rangle-  \langle  \eta, \nu \rangle =
  \langle  \xi , \eta \plus \nu  \rangle - \langle  \xi, \eta
  \rangle  
.   
  $$
\item
  The squared norm of  a pre-Hilbert space
  is defined through a non-trivial supremum of sums of scalar products
  (Example \ref{bsp:hilbert.2}).  

  
\item
  The Tsallis entropy $S_q$ \citep{tsallis88} is additive for all
  $q\in\RR$
  and not only for $q=1$, cf.~Example \ref{ex:tsallis}.

\item
  We make several material assumptions, in particular, we require the
  non-negativity of the entropy, of the hemi-metric, and of the
  hemi-scalar product for two identical arguments.
  As these axioms are along the same line, they look nice and
  promising.
  However, we do not claim
  that these axioms are compulsatory for any generalized approach.
  But we strongly suggest that the implicit choice of the underlying
  algebraic structure should enter more into the consciousness
  in any statistical modelling.
  
\item
  Finally, 
  a statistical model is not a set of distributions, but essentially
  a set of
  transformations with a certain algebraic structure.
  This point of view pursues those of 
  \cite{mccullagh2002statistical} and
  \cite{helland2004statistical}.
  \cite{schlatherreinbott21} show a range of potential
  applications of a such an algebraic definition of a statistical model,
  including linear models and PCA for exteme values.

\end{itemize}



\section*{Acknowledgment}
The author is grateful to
Christopher D\"orr,
Felix Reinbott,
and
 Wolfgang zu Castell 
for many hints and valuable discussions.
The author also thanks 
Carmen Ditscheid, 
Alexander Freudenberg, 
Tilmann Gneiting,
 Khadija Larabi,
 Peter Parczewski, 
 Achim Schlather,
 and
Milan Stehlik.




  \bibliographystyle{plainnat} 
  \bibliography{/home/schlather/tex/all,/home/schlather/tex/ms,bibliography}       

\end{document}